\documentclass[oneside]{memoir}

\usepackage{amsmath}
\usepackage{amssymb}
\usepackage{amsthm}
\usepackage[doi=false, isbn=false, giveninits=true, url=false]{biblatex}
\usepackage{hyperref}
\usepackage{microtype}
\usepackage{thmtools}
\usepackage{tikz}
\usepackage{tikz-cd}
\usepackage{xcolor}

\usepackage[capitalize]{cleveref}

\declaretheorem[numberwithin=section]{theorem}

\declaretheorem[sibling=theorem]{conjecture}
\declaretheorem[sibling=theorem]{corollary}
\declaretheorem[sibling=theorem]{lemma}
\declaretheorem[sibling=theorem]{question}
\declaretheorem[sibling=theorem]{remark}

\declaretheorem[sibling=theorem, style=definition]{definition}

\counterwithout{section}{chapter}
\numberwithin{figure}{section}

\newcommand\blfootnote[1]{%
  \begingroup
  \renewcommand\thefootnote{}\footnote{#1}%
  \addtocounter{footnote}{-1}%
  \endgroup
}

\title{Simplicial complexes from finite projective planes and colored configurations}
\author{Matt Superdock\\\\
\emph{Department of Mathematics and Computer Science}\\
\emph{Rhodes College, Memphis, TN 38112, USA}}
\date{\vspace{-5ex}}

\begin{document}

\maketitle

\blfootnote{\emph{E-mail address:}
  \href{mailto:superdockm@rhodes.edu}{superdockm@rhodes.edu}}
\blfootnote{\emph{Date:}
  March 24, 2022}
\blfootnote{Supported by NSF grant DMS 1855591.}

\begin{abstract}
  In the 7-vertex triangulation of the torus, the 14 triangles can be
  partitioned as $T_{1} \sqcup T_{2}$, such that each $T_{i}$ represents the
  lines of a copy of the Fano plane $PG(2, \mathbb{F}_{2})$. We generalize this
  observation by constructing, for each prime power $q$, a simplicial complex
  $X_{q}$ with $q^{2} + q + 1$ vertices and $2(q^{2} + q + 1)$ facets consisting
  of two copies of $PG(2, \mathbb{F}_{q})$.

  Our construction works for any \emph{colored $k$-configuration}, defined as a
  $k$-configuration whose associated bipartite graph $G$ is connected and has a
  $k$-edge coloring $\chi \colon E(G) \to [k]$, such that for all $v \in V(G)$,
  $a, b, c \in [k]$, following edges of colors $a, b, c, a, b, c$ from $v$
  brings us back to $v$. We give one-to-one correspondences between (1) Sidon
  sets of order 2 and size $k + 1$ in groups with order $n$, (2) linear codes
  with radius 1 and index $n$ in the lattice $A_{k}$, and (3) colored $(k +
  1)$-configurations with $n$ points and $n$ lines. (The correspondence between (1) and (2) is known.) As a result, we suggest
  possible topological obstructions to the existence of Sidon sets, and in
  particular, planar difference sets.
\end{abstract}

\section{Introduction}

The torus $T^{2}$ has a 7-vertex triangulation, arising from \cref{fig:torus}.

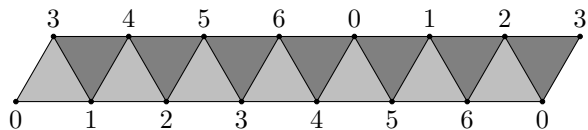
\begin{figure}
\begin{center}
\begin{tikzpicture}
  \path[fill = gray!50] (0, 0) -- (1, 0) -- (0.5, 0.866) -- (0, 0);
  \path[fill = gray!50] (1, 0) -- (2, 0) -- (1.5, 0.866) -- (1, 0);
  \path[fill = gray!50] (2, 0) -- (3, 0) -- (2.5, 0.866) -- (2, 0);
  \path[fill = gray!50] (3, 0) -- (4, 0) -- (3.5, 0.866) -- (3, 0);
  \path[fill = gray!50] (4, 0) -- (5, 0) -- (4.5, 0.866) -- (4, 0);
  \path[fill = gray!50] (5, 0) -- (6, 0) -- (5.5, 0.866) -- (5, 0);
  \path[fill = gray!50] (6, 0) -- (7, 0) -- (6.5, 0.866) -- (6, 0);

  \path[fill = gray] (0.5, 0.866) -- (1, 0) -- (1.5, 0.866) -- (0.5, 0.866);
  \path[fill = gray] (1.5, 0.866) -- (2, 0) -- (2.5, 0.866) -- (1.5, 0.866);
  \path[fill = gray] (2.5, 0.866) -- (3, 0) -- (3.5, 0.866) -- (2.5, 0.866);
  \path[fill = gray] (3.5, 0.866) -- (4, 0) -- (4.5, 0.866) -- (3.5, 0.866);
  \path[fill = gray] (4.5, 0.866) -- (5, 0) -- (5.5, 0.866) -- (4.5, 0.866);
  \path[fill = gray] (5.5, 0.866) -- (6, 0) -- (6.5, 0.866) -- (5.5, 0.866);
  \path[fill = gray] (6.5, 0.866) -- (7, 0) -- (7.5, 0.866) -- (6.5, 0.866);

  \draw[fill] (0, 0) circle [radius = 0.03];
  \draw[fill] (1, 0) circle [radius = 0.03];
  \draw[fill] (2, 0) circle [radius = 0.03];
  \draw[fill] (3, 0) circle [radius = 0.03];
  \draw[fill] (4, 0) circle [radius = 0.03];
  \draw[fill] (5, 0) circle [radius = 0.03];
  \draw[fill] (6, 0) circle [radius = 0.03];
  \draw[fill] (7, 0) circle [radius = 0.03];

  \draw[fill] (0.5, 0.866) circle [radius = 0.03];
  \draw[fill] (1.5, 0.866) circle [radius = 0.03];
  \draw[fill] (2.5, 0.866) circle [radius = 0.03];
  \draw[fill] (3.5, 0.866) circle [radius = 0.03];
  \draw[fill] (4.5, 0.866) circle [radius = 0.03];
  \draw[fill] (5.5, 0.866) circle [radius = 0.03];
  \draw[fill] (6.5, 0.866) circle [radius = 0.03];
  \draw[fill] (7.5, 0.866) circle [radius = 0.03];

  \draw (0, 0) -- (1, 0) -- (2, 0) -- (3, 0)
    -- (4, 0) -- (5, 0) -- (6, 0) -- (7, 0);
  \draw (0.5, 0.866) -- (1.5, 0.866) -- (2.5, 0.866) -- (3.5, 0.866)
    -- (4.5, 0.866) -- (5.5, 0.866) -- (6.5, 0.866) -- (7.5, 0.866);
  \draw (0, 0) -- (0.5, 0.866) -- (1, 0) -- (1.5, 0.866)
    -- (2, 0) -- (2.5, 0.866) -- (3, 0) -- (3.5, 0.866)
    -- (4, 0) -- (4.5, 0.866) -- (5, 0) -- (5.5, 0.866)
    -- (6, 0) -- (6.5, 0.866) -- (7, 0) -- (7.5, 0.866);

  \node [below] at (0, 0) {$0$};
  \node [below] at (1, 0) {$1$};
  \node [below] at (2, 0) {$2$};
  \node [below] at (3, 0) {$3$};
  \node [below] at (4, 0) {$4$};
  \node [below] at (5, 0) {$5$};
  \node [below] at (6, 0) {$6$};
  \node [below] at (7, 0) {$0$};

  \node [above] at (0.5, 0.866) {$3$};
  \node [above] at (1.5, 0.866) {$4$};
  \node [above] at (2.5, 0.866) {$5$};
  \node [above] at (3.5, 0.866) {$6$};
  \node [above] at (4.5, 0.866) {$0$};
  \node [above] at (5.5, 0.866) {$1$};
  \node [above] at (6.5, 0.866) {$2$};
  \node [above] at (7.5, 0.866) {$3$};
\end{tikzpicture}
\caption{The 7-vertex triangulation of $T^{2}$.}
\label{fig:torus}
\end{center}
\end{figure}

To see this, identify any two vertices with the same label, and identify any two
edges whose ends have the same label. Combinatorially, this identification
produces a simplicial complex on 7 vertices. Topologically, this identification
is equivalent to first identifying the leftmost and rightmost edges to obtain a
cylinder, and then identifying the top and bottom circles to obtain $T^{2}$.

This triangulation $X$ of $T^{2}$ has several notable properties:
\begin{itemize}
  \item
  $X$ has exactly 7 vertices. (In fact, $X$ is vertex-minimal; any triangulation
  of $T^{2}$ has at least 7 vertices; see
  \cite{jungerman1980minimal,lutz2005triangulated}.)

  \item
  $X$ contains the Fano plane; the triangles pointing up,
  $$\{0, 1, 3\}, \{1, 2, 4\}, \{2, 3, 5\}, \{3, 4, 6\}, \{4, 5, 0\}, \{5, 6,
  1\}, \{6, 0, 2\},$$
  can be viewed as the lines of the Fano plane on points $\{0, \ldots , 6\}$ (see \cref{fig:fano}).
	\begin{figure}
  \begin{center}
  \begin{tikzpicture}[x = 2cm, y = 2cm]
    \draw[fill] (0, 1.732) circle [radius = 0.03];
    \draw[fill] (-0.5, 0.866) circle [radius = 0.03];
    \draw[fill] (0.5, 0.866) circle [radius = 0.03];
    \draw[fill] (0, 0.577) circle [radius = 0.03];
    \draw[fill] (-1, 0) circle [radius = 0.03];
    \draw[fill] (0, 0) circle [radius = 0.03];
    \draw[fill] (1, 0) circle [radius = 0.03];

    \draw (0, 1.732) -- (-1, 0);
    \draw (0, 1.732) -- (0, 0);
    \draw (0, 1.732) -- (1, 0);
    \draw (-1, 0) -- (0.5, 0.866);
    \draw (1, 0) -- (-0.5, 0.866);
    \draw (-1, 0) -- (1, 0);
    \draw (0, 0.577) circle [radius = 0.577];

    \node [above] at (0, 1.732) {$0$};
    \node [above left] at (-0.5, 0.866) {$1$};
    \node [above right] at (0.5, 0.866) {$2$};
    \node [below left] at (-1, 0) {$3$};
    \node [below] at (0, 0) {$4$};
    \node [below left] at (0, 0.5) {$5$};
    \node [below right] at (1, 0) {$6$};
  \end{tikzpicture}
  \end{center}
	\caption{The Fano plane.}
	\label{fig:fano}
	\end{figure}
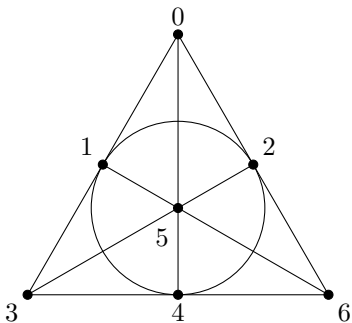
  (Similarly for the triangles pointing down.)

  \item
  $X$ is cyclic; the cyclic group $\mathbb{Z}_{7}$ acts on $X$ by cyclically
  permuting labels.

  \item
  $X$ is 2-neighborly; that is, each pair of vertices in $X$ form an edge.

\end{itemize}

The Fano plane $PG(2, \mathbb{F}_{2})$ is just one of a family of finite
projective planes $PG(2, \mathbb{F}_{q})$ for prime powers $q$. Hence we ask:

\begin{question}
  Does the 7-vertex triangulation of $T^{2}$, along with its notable properties,
  generalize to all prime power dimensions $q$?
\end{question}

Our main result is a construction (\cref{quotient}), which takes as its input a
colored $(k + 1)$-configuration $\mathcal{C}$ (see \cref{colored}), and produces
a simplicial complex $X(\mathcal{C})$ with $\pi_{1}(X(\mathcal{C})) \cong
\mathbb{Z}^{k}$. The complex $X(\mathcal{C})$ is not $T^{k}$ in general, but can
be made homeomorphic to $T^{k}$ by adding vertices and faces.

For example, if $\mathcal{C}$ is $PG(2, \mathbb{F}_{q})$, then we obtain the
following:

\begin{restatable*}{corollary}{pgcomplex}
\label{pg-complex}
  Let $q$ be a prime power. Then there exists a connected $q$-dimensional
  simplicial complex $X_{q}$ with $\pi_{1}(X_{q}) \cong \mathbb{Z}^{q}$, such
  that:
  \begin{itemize}
    \item
    $X_{q}$ has exactly $q^{2} + q + 1$ vertices.

    \item
    $X_{q}$ contains two copies of $PG(2, \mathbb{F}_{q})$, each consisting of
    $q^{2} + q + 1$ facets of $X_{q}$. These two copies fully describe $X_{q}$,
    in that these $2(q^{2} + q + 1)$ facets are all of the facets of $X_{q}$,
    and every face of $X_{q}$ is contained in a facet.

    \item
    $X_{q}$ is cyclic; the cyclic group $\mathbb{Z}_{q^{2} + q + 1}$ acts freely
    on $X_{q}$.

    \item
    $X_{q}$ is 2-neighborly.

  \end{itemize}
\end{restatable*}

In the general construction, the complex $X(\mathcal{C})$ contains one copy of
$\mathcal{C}$, and one copy of its dual $\mathcal{C}^{*}$, which is obtained
from $\mathcal{C}$ by switching points and lines. Since $PG(2, \mathbb{F}_{q})$
is isomorphic to its dual, in this case we obtain two copies of $PG(2,
\mathbb{F}_{q})$. See \cref{construction-planar} for details.

We make no claim of vertex-minimality. However, we note that the smallest known
triangulations of $T^{k}$ use $2^{k + 1} - 1$ vertices
(\cite{kuhnel1988combinatorial}; see~\cite{lutz2005triangulated}). Our complex
$X_{q}$ from \cref{pg-complex} uses fewer vertices but lacks the full structure
of $T^{q}$. We conjecture that $X_{q}$ is vertex-minimal in the following sense:

\begin{restatable*}{conjecture}{minimumcyclic}
\label{minimum-cyclic}
  Suppose $X$ is a simplicial complex on $n$ vertices, such that $\pi_{1}(X)
  \cong \mathbb{Z}^{k}$, and such that $X$ admits a free
  $\mathbb{Z}_{n}$-action. Then $n \ge k^{2} + k + 1$, with equality attainable
  only for prime powers $k$.
\end{restatable*}

This conjecture (along with \cref{planar-colored}) implies that every cyclic
planar difference set has prime power order, an open problem in design theory
(see~\cite{beth1999design}, Chapter~VII). In this way, our work suggests
possible topological obstructions to the existence of planar difference sets and
finite projective planes.

Our construction is closely related to a construction of linear codes from Sidon
sets (see~\cite{kovacevic2019sidon}). In our construction, we assign labels to a
$k$-dimensional lattice and then take a quotient according to that labeling. The
labeling of the lattice can be viewed as a linear code in the lattice $A_{k} =
\{\vec{x} \in \mathbb{Z}^{k + 1} : \sum_{i}x_{i} = 0\}$. This observation leads
to the following correspondences:
\begin{restatable*}{theorem}{correspondence}
\label{correspondence}
  We have one-to-one correspondences (up to isomorphism) between any two of the
  following three structures:
  \begin{enumerate}[(1)]
    \item
    Pairs $(G, B)$, where $G$ is an abelian group with $|G| = n$, and $B$ is a
    Sidon set of order 2 in $G$ with $|B| = k + 1$.

    \item
    Linear codes $\mathcal{L}$ with radius 1 in $A_{k}$, with $|A_{k} /
    \mathcal{L}| = n$.

    \item
    Colored $(k + 1)$-configurations $\mathcal{C}$ with $n$ points and $n$
    lines.

  \end{enumerate}
\end{restatable*}
The relationship between the first two structures above is known
\cite{kovacevic2019sidon}; our contribution is to introduce the third. This
raises the possibility of topological obstructions to the existence of Sidon
sets and linear codes, via the simplicial complex~$X(\mathcal{C})$ given by
\cref{quotient}.

We also obtain restatements of open problems on planar difference sets. For
example, the conjecture that all planar difference sets are cyclic becomes:

\begin{restatable*}{conjecture}{cyclic}
  Let $\mathcal{C}$ be a colored $k$-configuration which is also a projective
  plane. Then $\mathcal{C}$ admits a free cyclic group action.
\end{restatable*}

\section{Colored $k$-configurations}

Following Gr{\"u}nbaum \cite{grunbaum2009configurations}, we define a
$k$-configuration as a certain kind of $k$-regular incidence structure (where $k
\in \mathbb{N}$):

\begin{definition}
  A \emph{$k$-configuration} consists of finite sets $P, L$ (whose elements are
  called ``points'' and ``lines,'' respectively) and an incidence relation $R
  \subseteq P \times L$, satisfying the following conditions:

  \begin{enumerate}[(1)]
    \item
    There do not exist distinct $p_{1}, p_{2} \in P$ and distinct $l_{1}, l_{2}
    \in L$ with $(p_{i}, l_{j}) \in R$ for all $i, j \in \{1, 2\}$. (That is, no
    two points are on more than one common line; equivalently, no two lines
    contain more than one common point.)

    \item
    Each point is incident with exactly $k$ lines.

    \item
    Each line is incident with exactly $k$ points.

  \end{enumerate}
  For a $k$-configuration $\mathcal{C}$, we will often write $P, L, R$ as
  $P(\mathcal{C}), L(\mathcal{C}), R(\mathcal{C})$.
\end{definition}

A simple counting argument shows the following:

\begin{remark}
  In a $k$-configuration, the numbers of points and lines are equal.
\end{remark}

A $k$-configuration $\mathcal{C}$ corresponds to a $k$-regular bipartite graph
$G(\mathcal{C})$:

\begin{definition}
  Given a $k$-configuration $\mathcal{C}$, we define a graph $G(\mathcal{C})$:

  \begin{itemize}
    \item
    The vertex set of $G(\mathcal{C})$ is the disjoint union $P(\mathcal{C})
    \sqcup L(\mathcal{C})$.

    \item
    The edge set of $G(\mathcal{C})$ consists of edges $\{p, l\}$ for each $(p,
    l) \in R(\mathcal{C})$.

  \end{itemize}
\end{definition}

In fact, a $k$-configuration $\mathcal{C}$ is equivalent to a $k$-regular
bipartite graph $G(\mathcal{C})$ with no 4-cycle. The ``no 4-cycle'' requirement
corresponds to condition (1), and the $k$-regularity corresponds to conditions
(2) and (3). We will often require $G(\mathcal{C})$ to be connected:

\begin{definition}
  A $k$-configuration $\mathcal{C}$ is \emph{connected} if $G(\mathcal{C})$ is
  connected.
\end{definition}

To show that $\mathcal{C}$ is connected, it suffices to consider just the points
of $\mathcal{C}$:

\begin{lemma}
\label{connected}
  Let $\mathcal{C}$ be a $k$-configuration with $k > 0$, and suppose there
  exists a path in $G(\mathcal{C})$ between any two points of $\mathcal{C}$.
  Then $\mathcal{C}$ is connected.
\end{lemma}

\begin{proof}
  Since $k > 0$, each line in $\mathcal{C}$ is incident with at least one point.
\end{proof}

We now consider $k$-coloring the incidences of a $k$-configuration
$\mathcal{C}$; that is, we consider functions $\chi \colon R(\mathcal{C}) \to
[k]$, such that if two incidences $(p_{1}, l_{1}), (p_{2}, l_{2})$ have $p_{1} =
p_{2}$ or $l_{1} = l_{2}$, then $\chi(p_{1}, l_{1}) \ne \chi(p_{2}, l_{2})$.
This corresponds exactly to a \emph{$k$-edge coloring} of $G(\mathcal{C})$:

\begin{definition}
  For a graph $G$, a \emph{$k$-edge coloring} of $G$ is a function $\chi \colon
  E(G) \to [k]$ that assigns distinct colors to distinct edges sharing a vertex.
\end{definition}
  
Note that if $G$ is a $k$-regular, bipartite graph, then $G$ always has a
$k$-edge coloring, since line graphs of bipartite graphs are
perfect~\cite{konig1916graphen}. For our construction, we will need the $k$-edge
coloring $\chi$ to satisfy an additional property:

\begin{definition}
\label{6-cycle}
  Let $G$ be a $k$-regular bipartite graph, and let $\chi$ be a $k$-edge
  coloring of $G$. For each vertex $v \in V(G)$ and each color $c \in [k]$,
  let $\phi_{c}(v)$ be the unique vertex $w \in V(G)$ such that $vw$ is an
  edge of $G$ with color $\chi(vw) = c$.

  \medskip

  We say that $\chi$ has the \emph{6-cycle property}, if, for each vertex $v \in
  V(G)$, and for each triple $(a, b, c)$ of distinct colors $a, b, c \in [k]$,
  we have
  $$(\phi_{c} \circ \phi_{b} \circ \phi_{a} \circ \phi_{c} \circ \phi_{b} \circ
  \phi_{a})(v) = v.$$
  (This equation means that following the unique path from $v$ along edges with
  colors $a, b, c, a, b, c$, in order, produces a 6-cycle in $G$.)
\end{definition}

\begin{definition}
\label{colored}
  A \emph{colored $k$-configuration} is a connected $k$-configuration
  $\mathcal{C}$, along with a $k$-edge coloring $\chi(\mathcal{C})$ of
  $G(\mathcal{C})$ satisfying the 6-cycle property.
\end{definition}

Our construction (\cref{quotient}) takes a colored $k$-configuration as its
input; the construction involves assigning labels to the lattice~$A_{n}$, and
the 6-cycle property ensures that we assign labels consistently. In
\cref{construction-planar,construction-semifields}, we give several examples of
colored $k$-configurations. For example, $PG(2, \mathbb{F}_{q})$, which is a $(q
+ 1)$-configuration, can be colored (\cref{pg-colored}).

\subsection{Duality}

The duality on $k$-configurations extends to colored $k$-configurations:

\begin{definition}
  Let $\mathcal{C}$ be a colored $k$-configuration. Then we obtain a \emph{dual}
  colored $k$-configuration $\mathcal{C}^{*}$ as follows:
  \begin{itemize}
    \item
    The points of $\mathcal{C}^{*}$ are the lines of $\mathcal{C}$.

    \item
    The lines of $\mathcal{C}^{*}$ are the points of $\mathcal{C}$.

    \item
    A point $p$ and line $l$ are incident in $\mathcal{C}^{*}$ if $(l, p)$ is an
    incidence of $\mathcal{C}$.

    \item
    The color of the incidence $(p, l)$ in $\mathcal{C}^{*}$ is the color of
    $(l, p)$ in $\mathcal{C}$.

  \end{itemize}
\end{definition}

\subsection{Group actions}

Our next goal is to define group actions on colored $k$-configurations. As
usual, we begin by defining a notion of homomorphism:

\begin{definition}
  A \emph{homomorphism} $\psi \colon \mathcal{C} \to \mathcal{D}$ of colored
  $k$-configurations $\mathcal{C}, \mathcal{D}$ consists of functions $\psi_{P}
  \colon P(\mathcal{C}) \to P(\mathcal{D})$, $\psi_{L} \colon L(\mathcal{C}) \to
  L(\mathcal{D})$, $\psi_{\chi} \colon [k] \to [k]$, with the following
  properties:
  \begin{itemize}
    \item
    If $(p, l) \in R(\mathcal{C})$, then $(\psi_{P}(p), \psi_{L}(l)) \in
    R(\mathcal{D})$.

    \item
    If $(p, l) \in R(\mathcal{C})$, then $\chi(\mathcal{D})(\psi_{P}(p),
    \psi_{L}(l)) = \psi_{\chi}(\chi(\mathcal{C})(p, l))$.

  \end{itemize}
  In other words, $\psi_{P}, \psi_{L}$ preserve incidences and color classes of
  incidences.
\end{definition}

Then a group action $\rho$ by a group $G$ on a colored $k$-configuration
$\mathcal{C}$ is, as usual, a group homomorphism $\rho \colon G \to
\text{Aut}(\mathcal{C})$. Note that $\rho$ consists of group actions $\rho_{P}$
and $\rho_{L}$ by $G$ on $P(\mathcal{C}), L(\mathcal{C})$ preserving incidences
and color classes of incidences. We say that $\rho$ is free if $\rho_{P}$ is
free.

\begin{definition}
  Let $G$ be a group, and let $\mathcal{C}$ be a colored $k$-configuration. We
  say that a group action $\rho \colon G \to \text{Aut}(\mathcal{C})$ is
  \emph{free} if $\rho_{P} \colon G \to \text{Aut}(P(\mathcal{C}))$ is free.
\end{definition}

\noindent
Note that we do not require $\rho_{L}$ to be free.

\section{Finite projective planes}

In our definition of $k$-configuration, condition (1) says that any two points
lie on at most one common line, and vice versa. By replacing ``at most one''
with ``exactly one,'' we obtain a definition of a finite projective plane. We
recall the usual definition of a projective plane:

\begin{definition}[\cite{hall1943projective}]
\label{projective}
  A \emph{projective plane} consists of sets $P, L$ (whose elements are called
  ``points'' and ``lines,'' respectively), and an incidence relation $R
  \subseteq P \times L$, satisfying the following conditions:

  \begin{enumerate}[(1)]
    \item
    For distinct $p_{1}, p_{2} \in P$, there exists a unique $l \in L$ with
    $(p_{1}, l), (p_{2}, l) \in R$.
    
    \item
    For distinct $l_{1}, l_{2} \in L$, there exists a unique $p \in P$ with $(p,
    l_{1}), (p, l_{2}) \in R$.

    \item
    There exist distinct $p_{1}, p_{2}, p_{3}, p_{4} \in P$, such that for each
    $l \in L$, at most two of the four pairs $(p_{1}, l)$, $(p_{2}, l)$,
    $(p_{3}, l)$, $(p_{4}, l)$ are in $R$.

  \end{enumerate}

  We say that the projective plane is \emph{finite} if $P$ and $L$ are finite.
\end{definition}

In a finite projective plane, there exists an integer $q$, called the
\emph{order} of the finite projective plane, such that each point is incident
with exactly $q + 1$ lines, and each line is incident with exactly $q + 1$
points (see~\cite{veblen1907non}). It follows by a counting argument that the
plane has exactly $q^{2} + q + 1$ points and lines.

Note that any finite projective plane with order $q$ is a $(q +
1)$-configuration. Conversely, any $(q + 1)$-configuration with $q \ge 2$
satisfying conditions (1) and (2) of \cref{projective} is a finite projective plane; this follows
from Hall's characterization of degenerate planes \cite{hall1943projective}. We
also have the following characterization:

\begin{lemma}
\label{characterize-projective}
  A $(k + 1)$-configuration $\mathcal{C}$ with $n$ points and $n$ lines is a
  projective plane if and only if $n = k^{2} + k + 1$.
\end{lemma}

\begin{proof}
  Let $S$ be the set of triples $(p, q, \ell)$, where $p, q \in P(\mathcal{C})$
  are distinct points on $\ell \in L(\mathcal{C})$. Since any two points $p, q$
  lie on at most one common line, we have $|S| \le n(n - 1)$. Since any line
  $\ell$ contains exactly $k + 1$ points, we have $|S| = nk(k + 1)$. Therefore,
  $nk(k + 1) \le n(n - 1)$, so $n \ge k^{2} + k + 1$, with equality if and only
  if $\mathcal{C}$ is a projective plane.
\end{proof}

The standard example of a finite projective plane is $PG(2, \mathbb{F}_{q})$:

\begin{definition}
  Let $q$ be a prime power. Then $PG(2, \mathbb{F}_{q})$ is a finite projective
  plane of order $q$, defined with reference to the vector space
  $\mathbb{F}_{q}^{3}$ over $\mathbb{F}_{q}$:

  \begin{itemize}
    \item
    Let $P$ be the set of one-dimensional subspaces of $\mathbb{F}_{q}^{3}$.

    \item
    Let $L$ be the set of two-dimensional subspaces of $\mathbb{F}_{q}^{3}$.

    \item
    Let $R \subseteq P \times L$ be the set of pairs $(p, l)$ of subspaces $p,
    l$ with $p \subseteq l$.
  \end{itemize}

  (Conditions (1) and (2) follow from the identity $\dim(U + V) + \dim(U \cap V)
  = \dim U + \dim V$, and for condition (3) we may take the four one-dimensional
  subspaces spanned by $(1, 0, 0), (0, 1, 0), (0, 0, 1), (1, 1, 1)$.)
\end{definition}

It is an open problem whether there exists a finite projective plane whose
order $q$ is not a prime power. The strongest negative result is given by Bruck
\& Ryser~\cite{bruck1949nonexistence}: if the order $q$ of a finite projective
plane satisfies $q \equiv 1, 2\pmod{4}$, then $q$ is a sum of two squares.

\section{The lattice $A_{n}$ and the simplicial complex $W_{n}$}
\label{lattice}

We first introduce the lattice $A_{n}$, which is well known from the sphere
packing literature (see~\cite{conway1999sphere}):
\begin{definition}
  For $n \ge 0$, the lattice $A_{n}$ is defined by
  $$A_{n} = \left\{(x_{1}, \ldots , x_{n + 1}) \in \mathbb{Z}^{n + 1} : \sum_{i
  = 1}^{n + 1}x_{i} = 0\right\}.$$
\end{definition}

For example, $A_{2}$ is the hexagonal lattice, and $A_{3}$ is the face-centered
cubic lattice. Locally, $A_{n}$ has the structure of the expanded simplex
(see~\cite{coxeter1981derivation}). We may consider $A_{n}$ a metric space by
using a scaled $\ell_{1}$-metric, $d(\vec{x}, \vec{y}) = \|\vec{x} -
\vec{y}\|_{1} / 2$. (This metric $d$ can be viewed as a graph metric, where we
consider $\vec{x}, \vec{y}$ to be adjacent if $\|\vec{x} - \vec{y}\|_{1} = 2$.)

We now introduce a simplicial complex structure $W_{n}$ on $A_{n}$; the
simplicial complex $W_{n}$ is also known as the \emph{Rips complex} of diameter
1 of $A_{n}$~\cite{vietoris1927hoheren}.
\begin{definition}
  For $n \ge 0$, the simplicial complex $W_{n}$ is defined as follows:
  \begin{itemize}
    \item
    The vertices $V(W_{n})$ of $W_{n}$ are the points of the lattice $A_{n}$.

    \item
    The faces of $W_{n}$ are the sets $F$ with $d(\vec{x}, \vec{y}) = 1$ for all
    distinct $\vec{x}, \vec{y} \in F$.

  \end{itemize}
\end{definition}
Note that $W_{n}$ is a flag simplicial complex; that is, if $F$ is a set of
vertices of $W_{n}$, and all pairs of vertices in $F$ are edges of $W_{n}$, then
$F$ is a face of $W_{n}$. In this way, $W_{n}$ is determined by its 1-skeleton.
For $W_{2}$, the facets are exactly the triangles of the hexagonal lattice (see \cref{fig:w2}).

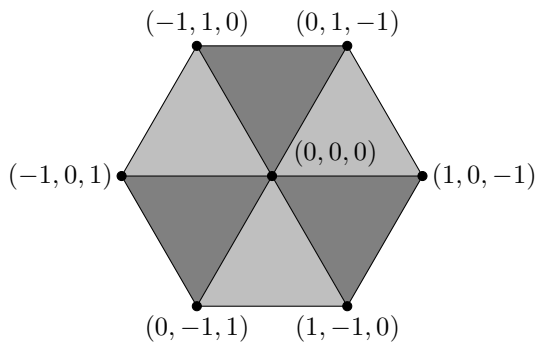
\begin{figure}
\begin{center}
\begin{tikzpicture}[x = 2cm, y = 2cm]
  \path[fill = gray] (0, 0) -- (0.5, 0.866) -- (-0.5, 0.866) -- (0, 0);
  \path[fill = gray] (0, 0) -- (-1, 0) -- (-0.5, -0.866) -- (0, 0);
  \path[fill = gray] (0, 0) -- (0.5, -0.866) -- (1, 0) -- (0, 0);

  \path[fill = gray!50] (0, 0) -- (1, 0) -- (0.5, 0.866) -- (0, 0);
  \path[fill = gray!50] (0, 0) -- (-0.5, 0.866) -- (-1, 0) -- (0, 0);
  \path[fill = gray!50] (0, 0) -- (-0.5, -0.866) -- (0.5, -0.866) -- (0, 0);

  \draw[fill] (-1, 0) circle [radius = 0.03];
  \draw[fill] (-0.5, -0.866) circle [radius = 0.03];
  \draw[fill] (-0.5, 0.866) circle [radius = 0.03];
  \draw[fill] (0, 0) circle [radius = 0.03];
  \draw[fill] (0.5, -0.866) circle [radius = 0.03];
  \draw[fill] (0.5, 0.866) circle [radius = 0.03];
  \draw[fill] (1, 0) circle [radius = 0.03];

  \draw (-1, 0) -- (0, 0) -- (1, 0);
  \draw (-0.5, -0.866) -- (0, 0) -- (0.5, 0.866);
  \draw (-0.5, 0.866) -- (0, 0) -- (0.5, -0.866);

  \draw (1, 0) -- (0.5, 0.866) -- (-0.5, 0.866) -- (-1, 0)
    -- (-0.5, -0.866) -- (0.5, -0.866) -- (1, 0);

  \node[above right] at (0.08, 0) {$(0, 0, 0)$};
  \node[right] at (1, 0) {$(1, 0, -1)$};
  \node[above] at (0.5, 0.866) {$(0, 1, -1)$};
  \node[above] at (-0.5, 0.866) {$(-1, 1, 0)$};
  \node[left] at (-1, 0) {$(-1, 0, 1)$};
  \node[below] at (-0.5, -0.866) {$(0, -1, 1)$};
  \node[below] at (0.5, -0.866) {$(1, -1, 0)$};
\end{tikzpicture}
\caption{The neighborhood of $(0, 0, 0)$ in $W_{2}$.}
\label{fig:w2}
\end{center}
\end{figure}
All six shaded triangles are facets of $W_{2}$; the light and dark shading
indicate ``positive'' and ``negative'' facets, respectively, as defined below.

To gain intuition for $A_{n}$ and $W_{n}$, consider the projection $p \colon
\mathbb{R}^{n + 1} \to \mathbb{R}^{n}$ onto the first $n$ coordinates, that is,
$$p(x_{1}, \ldots , x_{n + 1}) = (x_{1}, \ldots , x_{n}).$$
Then $p$ carries $A_{n}$ isomorphically (as an additive group) to
$\mathbb{Z}^{n}$, and embeds $W_{n}$ into $\mathbb{R}^{n}$. For example, in \cref{fig:projections} we draw
the facets of $p(W_{2})$ inside the unit square of $\mathbb{R}^{2}$, and the
facets of $p(W_{3})$ inside the unit cube of $\mathbb{R}^{3}$:

\begin{figure}
\begin{center}
\begin{tikzpicture}[x = 2cm, y = 2cm]
  \begin{scope}[yshift = 0.25cm]
    \path[fill = gray!50] (0, 0) -- (0, 1) -- (1, 0) -- (0, 0);
    \path[fill = gray] (0, 1) -- (1, 0) -- (1, 1) -- (0, 1);

    \draw[fill] (0, 0) circle [radius = 0.03];
    \draw[fill] (0, 1) circle [radius = 0.03];
    \draw[fill] (1, 0) circle [radius = 0.03];
    \draw[fill] (1, 1) circle [radius = 0.03];

    \draw (0, 0) -- (0, 1) -- (1, 1) -- (1, 0) -- (0, 0);
    \draw (0, 1) -- (1, 0);
  \end{scope}

  \begin{scope}[xshift = 5cm]
    \path[fill = gray!50] (0, 0) -- (0, 1) -- (1, 0) -- (0, 0);
    \path[fill = gray] (0.3, 1.5) -- (1.3, 0.5) -- (1.3, 1.5) -- (0.3, 1.5);

    \draw[fill] (0, 0) circle [radius = 0.03];
    \draw[fill] (0, 1) circle [radius = 0.03];
    \draw[fill] (1, 0) circle [radius = 0.03];
    \draw[fill] (1, 1) circle [radius = 0.03];

    \draw[fill] (0.3, 0.5) circle [radius = 0.03];
    \draw[fill] (0.3, 1.5) circle [radius = 0.03];
    \draw[fill] (1.3, 0.5) circle [radius = 0.03];
    \draw[fill] (1.3, 1.5) circle [radius = 0.03];

    \draw (0, 0) -- (0, 1) -- (1, 1) -- (1, 0) -- (0, 0);
    \draw (0, 1) -- (1, 0);

    \draw (0.3, 0.5) -- (0.3, 1.5) -- (1.3, 1.5) -- (1.3, 0.5) -- (0.3, 0.5);
    \draw (0.3, 1.5) -- (1.3, 0.5);

    \draw (0, 0) -- (0.3, 0.5);
    \draw (0, 1) -- (0.3, 1.5);
    \draw (1, 0) -- (1.3, 0.5);
    \draw (1, 1) -- (1.3, 1.5);

    \draw (0.3, 0.5) -- (0, 1);
    \draw (1.3, 0.5) -- (1, 1);
    \draw (0.3, 0.5) -- (1, 0);
    \draw (0.3, 1.5) -- (1, 1);
  \end{scope}
\end{tikzpicture}
\caption{The projections $p(W_{2})$ and $p(W_{3})$.}
\label{fig:projections}
\end{center}
\end{figure}
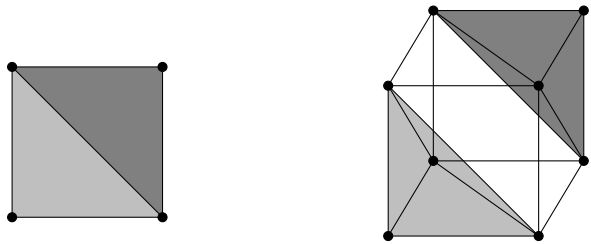
We see that $W_{2}$ is a tiling of the plane $\{(x, y, z) \in \mathbb{R}^{3} : x
+ y + z = 0\}$, but for $n > 2$, $W_{n}$ does not fill the hyperplane it spans.
(This is why our construction in \cref{quotient} will not generally give a
trianguation of $T^{n}$.)

\begin{lemma}
\label{facets}
  The simplicial complex $W_{n}$ has the following properties:

  \begin{itemize}
    \item
    $W_{n}$ has dimension $n$.

    \item
    The facets of $W_{n}$ are of the form
    $$\{\vec{x} + \vec{e}_{i}\}_{i \in [n + 1]}\qquad\text{or}\qquad
      \{\vec{x} - \vec{e}_{j}\}_{j \in [n + 1]}$$
    for fixed $\vec{x} \in \mathbb{Z}^{n + 1}$ with sum of coordinates $-1$ (in
    the first case) or $1$ (in the second). We call facets of the first type
    ``positive,'' and facets of the second type ``negative.'' We say the facet
    is ``rooted at $\vec{x}$'' in either case. (Here the $\vec{e}_{i}$ are the
    standard basis vectors in $\mathbb{R}^{n + 1}$.)

    \item
    Each vertex of $W_{n}$ is incident with exactly $n + 1$ facets of each type.

    \item
    Each face $F$ of $W_{n}$ is contained in a facet of $W_{n}$. If $\dim F \ge
    2$, then this facet is unique.
  \end{itemize}
\end{lemma}

\begin{proof}
  Let $F$ be a face of $W_{n}$, and assume $\vec{0} \in F$. Then each other
  vertex of $F$ is of the form $\vec{e}_{i} - \vec{e}_{j}$ for distinct $i, j
  \in [n + 1]$. Denote such a vertex by the ordered pair $(i, j)$. If $F$
  contains two nonzero vertices $(i_{1}, j_{1}), (i_{2}, j_{2})$, then $i_{1} =
  i_{2}$ or $j_{1} = j_{2}$. It follows that the nonzero vertices of $F$ are of
  the form
  $$\{(i_{1}, j), \ldots , (i_{k}, j)\} \qquad\text{or}\qquad
    \{(i, j_{1}), \ldots , (i, j_{k})\}.$$
  In the first case, we get a positive facet rooted at $-\vec{e}_{j}$; in the
  second case, we get a negative facet rooted at $\vec{e}_{i}$. The rest
  follows.
\end{proof}

\begin{lemma}
\label{action}
  The group action of $A_{n}$ on itself by addition induces a group action of
  $A_{n}$ on the simplicial complex $W_{n}$.
\end{lemma}

\begin{proof}
  By the translational symmetry of the definition of $W_{n}$.
\end{proof}

Recall that a simplicial complex $X$ is \emph{simply connected} if its geometric
realization is connected and has trivial fundamental group. We saw above that
the facets of $W_{3}$ fail to fill 3-dimensional space, but the projection
$p(W_{3})$ does contain the 2-skeleton of (a CW complex homeomorphic to)
$\mathbb{R}^{3}$. Hence $W_{3}$ is simply connected. This generalizes to $W_{n}$
for arbitrary $n$:

\begin{lemma}
\label{simply-connected}
  The simplicial complex $W_{n}$ is simply connected.
\end{lemma}

\begin{proof}
  The space $\mathbb{R}^{n}$ has a CW complex structure~$V_{n}$; for each
  $\vec{x} \in \mathbb{Z}^{n}$, $I \subseteq [n]$, the CW complex~$V_{n}$ has a
  $|I|$-cell given by
  $$f_{\vec{x}, I} = \left\{\vec{x} + \sum_{i \in I}\lambda_{i}\vec{e}_{i}
    : \lambda_{i} \in [0, 1]\right\}.$$
  
  We claim each 2-cell $f_{\vec{x}, \{i, j\}}$ is present in $p(W_{n})$ as the
  union of two triangles. To see this, let $\vec{v} \in A_{n}$ be the unique
  vector with $p(\vec{v}) = \vec{x}$. Then $f_{\vec{x}, \{i, j\}}$ lies in the
  image $p(W_{n})$ as shown in \cref{fig:2-cell}.

  \begin{figure}
  \begin{center}
  \begin{tikzpicture}[x = 2cm, y = 2cm]
    \begin{scope}
      \path[fill = gray!50] (0, 0) -- (0, 1) -- (1, 0) -- (0, 0);
      \path[fill = gray] (0, 1) -- (1, 0) -- (1, 1) -- (0, 1);

      \draw[fill] (0, 0) circle [radius = 0.03];
      \draw[fill] (0, 1) circle [radius = 0.03];
      \draw[fill] (1, 0) circle [radius = 0.03];
      \draw[fill] (1, 1) circle [radius = 0.03];

      \draw (0, 0) -- (0, 1) -- (1, 1) -- (1, 0) -- (0, 0);
      \draw (0, 1) -- (1, 0);

      \node[below left] at (0, 0)
        {$\vec{v}$};
      \node[above left] at (0, 1)
        {$\vec{v} + \vec{e}_{j} - \vec{e}_{n + 1}$};
      \node[below] at (1, 0)
        {$\vec{v} + \vec{e}_{i} - \vec{e}_{n + 1}$};
      \node[above] at (1, 1)
        {$\vec{v} + \vec{e}_{i} + \vec{e}_{j} - 2\vec{e}_{n + 1}$};
    \end{scope}

    \draw [->] (1.6, 0.5) -- (1.9, 0.5);

    \begin{scope}[xshift = 5cm]
      \path[fill = gray!50] (0, 0) -- (0, 1) -- (1, 0) -- (0, 0);
      \path[fill = gray] (0, 1) -- (1, 0) -- (1, 1) -- (0, 1);

      \draw[fill] (0, 0) circle [radius = 0.03];
      \draw[fill] (0, 1) circle [radius = 0.03];
      \draw[fill] (1, 0) circle [radius = 0.03];
      \draw[fill] (1, 1) circle [radius = 0.03];

      \draw (0, 0) -- (0, 1) -- (1, 1) -- (1, 0) -- (0, 0);
      \draw (0, 1) -- (1, 0);

      \node[below] at (0, 0)
        {$\vec{x}$};
      \node[above] at (0, 1)
        {$\vec{x} + \vec{e}_{j}$};
      \node[below right] at (1, 0)
        {$\vec{x} + \vec{e}_{i}$};
      \node[above right] at (1, 1)
        {$\vec{x} + \vec{e}_{i} + \vec{e}_{j}$};
    \end{scope}
  \end{tikzpicture}
	\caption{Each 2-cell is present in $p(W_{n})$.}
	\label{fig:2-cell}
  \end{center}
	\end{figure}
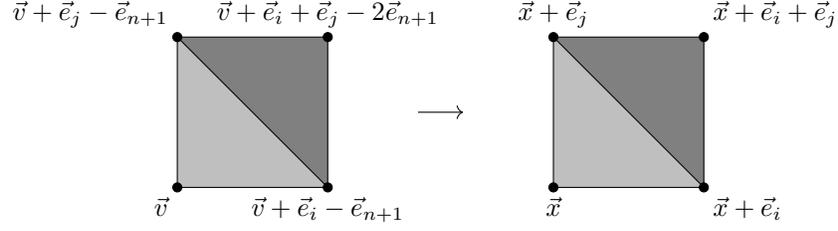
  Moreover, all 1-dimensional faces of $W_{n}$ appear in \cref{fig:2-cell} for some
  $\vec{x}, i, j$, so the 2-skeleton $\text{sk}_{2}(W_{n})$ can be obtained from
  $\text{sk}_{2}(V_{n})$ by attaching 2-faces. Since the fundamental group of a
  complex depends only on its 2-skeleton, we have $\pi_{1}(\text{sk}_{2}(V_{n}))
  = \pi_{1}(\mathbb{R}^{n}) = 0$, so $\pi_{1}(\text{sk}_{2}(W_{n})) = 0$, so
  $\pi_{1}(W_{n}) = 0$.
\end{proof}

For simplicial complexes, we may also define homotopy combinatorially:

\begin{definition}
\label{path}
  Let $X$ be a simplicial complex. A \emph{path} $\gamma$ in $X$ is a finite
  sequence of vertices $v_{0}, \ldots , v_{m}$, for some $m \ge 0$, with $v_{i},
  v_{i + 1}$ adjacent for each $i \in \{0, \ldots , m - 1\}$. We say that
  $\gamma$ \emph{starts at $v_{0}$} and \emph{ends at $v_{m}$}, and we may refer
  to $\gamma$ as a path \emph{from $v_{0}$} \emph{to $v_{m}$}.
\end{definition}

\begin{definition}
  Let $X$ be a simplicial complex, let $u, v, w \in V(X)$, let $\gamma$ be a
  path $v_{0}, \ldots , v_{m}$ in $X$ from $u$ to $v$, and let $\delta$ be a
  path $w_{0}, \ldots , w_{n}$ in $X$ from $v$ to $w$. Then we define $\gamma
  \cdot \delta$ as the path $v_{0}, \ldots , v_{m}, w_{1}, \ldots , w_{n}$ in
  $X$ from $u$ to $w$.
\end{definition}

\begin{definition}
  Let $X$ be a simplicial complex, and let $\gamma, \gamma'$ be paths in $X$. We
  define a relation $\gamma \simeq \gamma'$ (in words, ``$\gamma, \gamma'$ are
  homotopic'') inductively:
  \begin{enumerate}[(1)]
    \item 
    If $u, v \in V(X)$ are adjacent, then $u, v, u \simeq u$.

    \item
    If $u, v, w\in V(X)$ are contained in a common face $\{u, v, w\}$ of $X$,
    then $u, v, w \simeq u, w$.

    \item
    If $\gamma \simeq \gamma'$ and $\delta \simeq \delta'$, then $\gamma \cdot
    \delta \simeq \gamma' \cdot \delta'$.

    \item
    For any path $\gamma$, we have $\gamma \simeq \gamma$.

    \item
    If $\gamma \simeq \gamma'$, then $\gamma' \simeq \gamma$.

    \item
    If $\gamma \simeq \gamma'$ and $\gamma' \simeq \gamma''$, then $\gamma
    \simeq \gamma''$.

  \end{enumerate}

  (An induction on the definition shows that if $\gamma \simeq \gamma'$, then
  $\gamma, \gamma'$ start at the same vertex and end at the same vertex.)
\end{definition}

The usual topological notion of a simplicial complex being simply connected is
equivalent to a combinatorial notion using the definitions above (see
\cite{reynaud2003algebraic}). As a result, we get the following corollary to
\cref{simply-connected}:

\begin{corollary}
\label{simply-connected-cor}
  Let $\vec{u}, \vec{v} \in V(W_{n})$, and let $\gamma, \gamma'$ be paths in
  $W_{n}$ from $\vec{u}$ to $\vec{v}$. Then we have $\gamma \simeq \gamma'$.
\end{corollary}

With these preliminaries in hand, we are ready to give our construction.

\section[Simplicial complexes from colored $k$-configurations]{Construction of simplicial complexes from colored $k$-configurations}
\label{construction-section}

We now give a construction taking a colored $(k + 1)$-configuration and giving a
labeling of the complex $W_{n}$ defined in \cref{lattice}. After taking a
quotient of $W_{n}$ according to this labeling, we obtain a simplicial complex
with properties analogous to the 7-vertex triangulation of $T^{2}$
(\cref{quotient}).

\begin{lemma}
\label{construction}
  Let $\mathcal{C}$ be a colored $(k + 1)$-configuration. Then there exists a
  surjective labeling $\ell \colon V(W_{k}) \to P(\mathcal{C})$ of the vertices
  of $W_{k}$, such that:

  \begin{enumerate}[(1)]
    \item
    If $\vec{u}, \vec{v} \in V(W_{k})$ are adjacent, then $\ell(\vec{u}) \ne
    \ell(\vec{v})$.

    \item
    Given $\vec{u} \in V(W_{k})$ and a point $p \in P(\mathcal{C})$, there is at
    most one vertex $\vec{v} \in V(W_{k})$ adjacent to $\vec{u}$ with
    $\ell(\vec{v}) = p$. (If $\mathcal{C}$ is also a projective plane, then ``at
    most one'' can be replaced with ``exactly one.'')

    \item
    If $\vec{u}, \vec{v} \in V(W_{k})$ satisfy $\ell(\vec{u}) = \ell(\vec{v})$,
    then $\ell(\vec{u} + \vec{w}) = \ell(\vec{v} + \vec{w})$ for all $\vec{w}
    \in V(W_{k})$. (That is, $\ell$ is invariant under translation by $\vec{u} -
    \vec{v}$.)

  \end{enumerate}
\end{lemma}

\begin{proof}
  Given adjacent vertices $\vec{u}, \vec{v} \in V(W_{k})$ and a label
  $\ell(\vec{u})$, we will establish a rule for determining $\ell(\vec{v})$. We
  may uniquely write $\vec{v} = \vec{u} - \vec{e}_{i} + \vec{e}_{j}$ for
  distinct $i, j \in [k + 1]$; then our rule proceeds as follows:
  \begin{itemize}
    \item
    Let $l$ be the line of $\mathcal{C}$ incident with $p$ such that $(p, l)$
    has color $i$ in $\mathcal{C}$.

    \item
    Let $q$ be the point of $\mathcal{C}$ incident with $l$ such that $(q, l)$
    has color $j$ in $\mathcal{C}$.

    \item
    Take $\ell(\vec{v}) = q$.

  \end{itemize}
  We can also restate the rule using the functions $\phi_{c}$ from
  \cref{6-cycle}:
  $$\ell(\vec{v}) = (\phi_{j} \circ \phi_{i})(\ell(\vec{u}))$$
  (That is, in $G(\mathcal{C})$, following the unique path from the point
  $\ell(\vec{u})$ along edges with colors $i, j$, in order, brings us to
  $\ell(\vec{v})$.)

  \bigskip

  Now let $\gamma$ be a path in $W_{k}$ from $\vec{u}$ to $\vec{v}$, and let $p
  \in P(\mathcal{C})$. Consider assigning $\ell(\vec{u}) = p$, and then
  successively applying the rule above on each pair of consecutive vertices of
  $\gamma$, until we obtain a label $\ell(\vec{v}) = q$ at the end of $\gamma$.
  We define $\ell_{\gamma} \colon P(\mathcal{C}) \to P(\mathcal{C})$ by
  $\ell_{\gamma}(p) = q$. We claim that $\ell_{\gamma}$ respects path homotopy:

  \bigskip

  \noindent
  \emph{Claim.} Let $\gamma, \gamma'$ be paths in $W_{k}$. If $\gamma \simeq
  \gamma'$, then $\ell_{\gamma} = \ell_{\gamma'}$.

  \bigskip

  \noindent
  \emph{Proof of claim.} We induct on the definition of $(\simeq)$. Cases (3),
  (4), (5), (6) are clear, so we turn to (1), (2):

  \bigskip

  \noindent
  \emph{Case (1).} $\vec{u}, \vec{v} \in V(W_{k})$ are adjacent; $\gamma$ is
  $\vec{u}, \vec{v}, \vec{u}$, and $\gamma'$ is $\vec{u}$.

  \bigskip

  \noindent
  \emph{Proof of case (1).} Write $\vec{v} = \vec{u} - \vec{e}_{i} +
  \vec{e}_{j}$ as above. Then we have
  $$\ell_{\gamma} = \phi_{i} \circ \phi_{j} \circ \phi_{j} \circ \phi_{i} =
  \phi_{i} \circ \phi_{i} = 1_{P(\mathcal{C})} = \ell_{\gamma'},$$
  where $1_{P(\mathcal{C})}$ is the identity function on $P(\mathcal{C})$.

  \bigskip

  \noindent
  \emph{Case (2).} $\vec{u}, \vec{v}, \vec{w} \in V(W_{k})$ are pairwise
  adjacent; $\gamma$ is $\vec{u}, \vec{v}, \vec{w}$, and $\gamma'$ is $\vec{u},
  \vec{w}$.

  \bigskip

  \noindent
  \emph{Proof of case (2).} By \cref{facets}, $\vec{u}, \vec{v}, \vec{w}$ are
  contained in a unique facet $F$ of $W_{k}$. On one hand, if $F$ is positive
  (see \cref{facets}), then we have
  $$\vec{u} = \vec{x} + \vec{e}_{i},\qquad
    \vec{v} = \vec{x} + \vec{e}_{j},\qquad
    \vec{w} = \vec{x} + \vec{e}_{k'}.$$
  As a result, we have
  $$\ell_{\gamma} = \phi_{k'} \circ \phi_{j} \circ \phi_{j} \circ \phi_{i} =
  \phi_{k'} \circ \phi_{i} = \ell_{\gamma'}.$$

  On the other hand, if $F$ is negative, then we have
  $$\vec{u} = \vec{x} - \vec{e}_{i},\qquad
    \vec{v} = \vec{x} - \vec{e}_{j},\qquad
    \vec{w} = \vec{x} - \vec{e}_{k'}.$$
  As a result, we have
  $$\ell_{\gamma} = \phi_{j} \circ \phi_{k'} \circ \phi_{i} \circ
  \phi_{j},\qquad \ell_{\gamma'} = \phi_{i} \circ \phi_{k'}.$$
  Since $\phi_{i}, \phi_{j}, \phi_{k'}$ are involutions, we have
  $\ell_{\gamma'}^{-1} = \phi_{k'} \circ \phi_{i}$, so
  $$\ell_{\gamma} \circ \ell_{\gamma'}^{-1} = \phi_{j} \circ \phi_{k'} \circ
  \phi_{i} \circ \phi_{j} \circ \phi_{k'} \circ \phi_{i} = 1_{P}$$
  by the 6-cycle property. Precomposing by $\ell_{\gamma'}$ gives $\ell_{\gamma}
  = \ell_{\gamma'}$. This completes the proof of the case and the claim.

  \bigskip

  We now define $\ell$. Fix $\vec{v}_{0} \in V(W_{k})$ and $p_{0} \in
  P(\mathcal{C})$, and for each $\vec{v} \in V(W_{k})$ and each path $\gamma$
  from $\vec{v}_{0}$ to $\vec{v}$, define $\ell(\vec{v}) =
  \ell_{\gamma}(p_{0})$. Since $W_{k}$ is connected, there is at least one such
  $\gamma$. Also, $\ell(\vec{v})$ does not depend on $\gamma$, since if $\gamma,
  \gamma'$ are two paths from $\vec{v}_{0}$ to $\vec{v}$, then $\gamma \simeq
  \gamma'$ by \cref{simply-connected-cor}, so $\ell_{\gamma} = \ell_{\gamma'}$
  by the claim above. Since $\mathcal{C}$ is connected, $\ell$ is surjective.

  To show that property (1) holds, let $\gamma$ be a path in $W_{k}$ from
  $\vec{v}_{0}$ to $\vec{u}$, and consider the path $\gamma, \vec{v}$ obtained
  by appending $\vec{v}$ to $\gamma$. Then write $\vec{v} = \vec{u} -
  \vec{e}_{i} + \vec{e}_{j}$ as above, where $i \ne j$. We have 
  $$\ell(\vec{v}) = \ell_{\gamma, \vec{v}}(p_{0}) = (\phi_{j} \circ
  \phi_{i})(\ell_{\gamma}(p_{0})) = (\phi_{j} \circ \phi_{i})(\ell(\vec{u})).$$
  It suffices to prove that $(\phi_{j} \circ \phi_{i})(p) \ne p$ for all $p \in
  P(\mathcal{C})$. This is clear in $G(\mathcal{C})$; if we follow edges of
  distinct colors $i, j$, in order, we cannot arrive back at the starting
  vertex.

  To show that property (2) holds, let $\gamma$ be a path in $W_{k}$ from
  $\vec{v}_{0}$ to $\vec{u}$. The neighbors of $\vec{u}$ in $W_{k}$ are
  $\vec{v}_{i, j} = \vec{u} - \vec{e}_{i} + \vec{e}_{j}$ for $i, j \in [k + 1]$,
  $i \ne j$. Therefore, $\ell(\vec{v}_{i, j}) = (\phi_{j} \circ
  \phi_{i})(\ell(\vec{u}))$; equivalently, in $G(\mathcal{C})$,
  $\ell(\vec{v}_{i, j})$ is obtained from $\ell(\vec{u})$ by following edges of
  colors $i, j$, in order. Any $i, j$ with $\ell(\vec{v}_{i, j}) = p$ correspond
  to a unique line incident with $\ell(\vec{u})$ and $p$. But $\ell(\vec{u}), p$
  lie on at most one common line, so there is at most one vertex $\vec{v}_{i,
  j}$ with $\ell(\vec{v}_{i, j}) = p$. (For a projective plane $\mathcal{C}$,
  replace ``at most'' with ``exactly.'')

  To show that property (3) holds, first suppose $\vec{w}$ is adjacent to the
  zero vector, say $\vec{w} = -\vec{e}_{i} + \vec{e}_{j}$. Let $\gamma$ be a
  path in $W_{k}$ from $\vec{v}_{0}$ to $\vec{u}$. Then
  $$\ell(\vec{u} + \vec{w})
    = \ell_{\gamma, \vec{w}}(p_{0})
    = (\phi_{j} \circ \phi_{i})(\ell_{\gamma}(p_{0}))
    = (\phi_{j} \circ \phi_{i})(\ell(\vec{u})).$$
  Hence if $\ell(\vec{u}) = \ell(\vec{v})$, then $\ell(\vec{u} + \vec{w}) =
  \ell(\vec{v} + \vec{w})$. We obtain property (3) for all $\vec{w}$ by
  induction on the distance $d(\vec{0}, \vec{w})$ in $W_{k}$. This completes the
  proof.
\end{proof}

The construction in \cref{construction} depends on choices of $\vec{v}_{0} \in
A_{k}$, $p_{0} \in P(\mathcal{C})$. We now show that these choices are
immaterial:

\begin{lemma}
\label{construction-choices}
  Let $\mathcal{C}$ be a colored $(k + 1)$-configuration, let $\vec{v}_{0},
  \vec{v}_{0}' \in A_{k}$, and let $p_{0}, p_{0}' \in P(\mathcal{C})$. Let
  $\ell$ be the labeling given by \cref{construction} using $\vec{v}_{0},
  p_{0}$, and let $\ell'$ be the labeling given by \cref{construction} using
  $\vec{v}_{0}', p_{0}'$. There exists a permutation $\sigma \colon
  P(\mathcal{C}) \to P(\mathcal{C})$ with $\ell' = \sigma \circ \ell$.
\end{lemma}

\begin{proof}
  Take $\vec{w}$ with $\ell'(\vec{0}) = \ell(\vec{w})$ by the surjectivity of
  $\ell$. Then $\ell'(\vec{v}) = \ell(\vec{v} + \vec{w})$ for all $\vec{v} \in
  V(W_{k})$, since the construction in \cref{construction} assigns
  $\ell(\vec{v})$ based only on an adjacent label $\ell(\vec{u})$ and indices
  $i, j$ determined from $\vec{v} - \vec{u}$.

  Now define $\sigma \colon P(\mathcal{C}) \to P(\mathcal{C})$ by
  $\sigma(\ell(\vec{v})) = \ell'(\vec{v})$ for all $\vec{v} \in V(W_{k})$. If
  $\ell(\vec{u}) = \ell(\vec{v})$, then $\ell'(\vec{u}) = \ell(\vec{u} +
  \vec{w}) = \ell(\vec{v} + \vec{w}) = \ell'(\vec{v})$ by property (3) of
  \cref{construction}, so $\sigma$ is well-defined. Since $\ell$ is surjective,
  $\sigma$ is a function on $P(\mathcal{C})$. Similarly, we may define an
  inverse function $\sigma^{-1} \colon P(\mathcal{C}) \to P(\mathcal{C})$ by
  $\sigma(\ell'(\vec{v})) = \ell(\vec{v})$, so $\sigma$ is a permutation. We
  have $\ell' = \sigma \circ \ell$ by the defining property of $\sigma$.
\end{proof}

Now we give our main result:

\begin{theorem}
\label{quotient}
  Let $\mathcal{C}$ be a colored $(k + 1)$-configuration with $n$ points and $n$
  lines. Then there exists a connected $k$-dimensional simplicial complex
  $X(\mathcal{C})$ with $\pi_{1}(X(\mathcal{C})) \cong \mathbb{Z}^{k}$, such
  that:
  \begin{enumerate}[(1)]
    \item
    $X(\mathcal{C})$ has exactly $n$ vertices.

    \item
    $X(\mathcal{C})$ contains a copy of $\mathcal{C}$, consisting of $n$ facets
    of $X(\mathcal{C})$, and a copy of the dual $(k + 1)$-configuration
    $\mathcal{C}^{*}$, consisting of $n$ other facets of $X(\mathcal{C})$. These
    two copies fully describe $X(\mathcal{C})$, in that these $2n$ facets are
    all the facets of $X(\mathcal{C})$, and every face of $X(\mathcal{C})$ is
    contained in a facet.

    \item
    Suppose $G$ acts on $\mathcal{C}$ via a group action $\rho$. Then $G$ acts
    on $X(\mathcal{C})$ via $\rho_{P}$, the action induced by $\rho$ on the
    points of $\mathcal{C}$.

    \item
    If $\mathcal{C}$ is also a projective plane, then $X(\mathcal{C})$ is
    2-neighborly.

  \end{enumerate}
\end{theorem}

\begin{proof}
  Apply \cref{construction} to obtain a labeling $\ell \colon V(W_{k}) \to
  P(\mathcal{C})$. Then define
  $$H = \{\vec{v} \in A_{k} : \ell(\vec{v}) = \ell(\vec{0})\}.$$
  For $\vec{u}, \vec{v} \in A_{k}$, with $\ell(\vec{u}), \ell(\vec{v}) = 0$,
  property~(3) of \cref{construction} implies
  $$\ell(\vec{u} + \vec{v}) = \ell(\vec{0} + \vec{v}) = \ell(\vec{v}) =
  \ell(\vec{0}).$$
  Likewise, property~(3) also implies
  $$\ell(-\vec{v}) = \ell(\vec{0} - \vec{v}) = \ell(\vec{v} - \vec{v}) =
  \ell(\vec{0}).$$
  Therefore, $H$ is an additive subgroup of $A_{k}$. Moreover, for any two
  $\vec{u}, \vec{v} \in A_{k}$, we have the following chain of equivalences:
  $$\ell(\vec{u}) = \ell(\vec{v})
    \quad\Leftrightarrow\quad \ell(\vec{u} - \vec{v}) = \ell(\vec{0})
    \quad\Leftrightarrow\quad \vec{u} - \vec{v} \in H.$$
  Hence the orbits of $H$ in $A_{k}$ correspond bijectively to points $p \in
  P(\mathcal{C})$ in the image of $\ell$. Since $\ell$ is surjective, there are
  exactly $n$ such orbits.

  We claim that $H$ spans the vector space $\{(x_{1}, \ldots , x_{k + 1}) \in
  \mathbb{R}^{k + 1} : \sum_{i}x_{i} = 0\}$. Suppose otherwise; then there
  exists $\vec{v} \in A_{k}$ with $\vec{v} \not\in \text{span }H$. Then we also
  have $\lambda\vec{v} \not\in \text{span }H$ for all nonzero $\lambda \in
  \mathbb{Z}$. But then $\ell(\lambda \vec{v})$ is distinct for distinct
  $\lambda \in \mathbb{Z}$, a contradiction since $P(\mathcal{C})$ is finite.
  Therefore, $H$ is a $k$-dimensional lattice, so we have $H \cong
  \mathbb{Z}^{k}$.

  By \cref{action}, $A_{k}$ acts on $W_{k}$ by addition, so $H$ does also.
  Therefore, the quotient $W_{k} / H$ is a well-defined $k$-dimensional CW
  complex, and in fact is a simplicial complex by~(1), (2) of
  \cref{construction}. Since $H$ is a ``covering space action'' in the
  sense of Hatcher~\cite{hatcher2002algebraic}, and $W_{k}$ is simply connected,
  we have $\pi_{1}(W_{k} / H) \cong H \cong \mathbb{Z}^{k}$
  (See~\cite{hatcher2002algebraic}, Proposition 1.40). Hence we may take
  $X(\mathcal{C}) = W_{k} / H$.

  It remains to show properties (1) through (4) from the theorem statement. Note
  that property (1) holds since there are $n$ orbits of $H$ in $A_{k}$, and
  property (4) holds by~(2) of \cref{construction}.

  Next, we show property (2). The facets of $X(\mathcal{C})$ corresponding to
  positive facets of $W_{k}$ comprise a copy of $\mathcal{C}$; here is a typical
  incidence of color $i$:
  \begin{align*}
    \text{positive facet rooted at $\vec{x}$ (mod $H$)} &\leftrightsquigarrow
      \text{line $\phi_{i}(\ell(\vec{x} + \vec{e}_{i}))$ of $\mathcal{C}$}\\
    \text{point $\vec{x} + \vec{e}_{i}$ (mod $H$)} &\leftrightsquigarrow
      \text{point $\ell(\vec{x} + \vec{e}_{i})$ of $\mathcal{C}$}
  \end{align*}
  To show that this correspondence is well-defined, note that
  \begin{align*}
    \phi_{i}(\ell(\vec{x} + \vec{e}_{i}))
      &= \phi_{i}((\phi_{i} \circ \phi_{j})(\ell(\vec{x} + \vec{e}_{j})))\\
      &= \phi_{j}(\ell(\vec{x} + \vec{e}_{j})).
  \end{align*}

  Next, the facets of $X(\mathcal{C})$ corresponding to negative facets of
  $W_{k}$ comprise a copy of $\mathcal{C}^{*}$; here is a typical incidence of
  color $i$:
  \begin{align*}
    \text{negative facet rooted at $\vec{x}$ (mod $H$)} &\leftrightsquigarrow
      \text{point $\ell(\vec{x} - \vec{e}_{k + 1})$ of $\mathcal{C}$}\\
    \text{point $\vec{x} - \vec{e}_{i}$ (mod $H$)} &\leftrightsquigarrow
      \text{line $\phi_{i}(\ell(\vec{x} - \vec{e}_{k + 1}))$ of $\mathcal{C}$}
  \end{align*}
  To show that this correspondence is well-defined, note that if $\vec{x} -
  \vec{e}_{i} = \vec{y} - \vec{e}_{j}$,
  \begin{align*}
    \phi_{j}(\ell(\vec{y} - \vec{e}_{k + 1}))
      &= \phi_{j}(\ell(\vec{x} - \vec{e}_{k + 1} - \vec{e}_{i} + \vec{e}_{j}))\\
      &= \phi_{j}((\phi_{j} \circ \phi_{i})(\ell(\vec{x} - \vec{e}_{k + 1})))\\
      &= \phi_{i}(\ell(\vec{x} - \vec{e}_{k + 1})).
  \end{align*}
  (Note that the index $k + 1$ here could be replaced by any index in $[k +
  1]$.)

  To show that property (3) holds, it suffices to prove that $\rho_{P}$ maps
  facets of $X(\mathcal{C})$ to other facets. This holds for positive facets,
  since these correspond to lines of $\mathcal{C}$ by (2). Since $\rho$ respects
  color classes, $\rho_{P}$ also respects the orientation of positive facets.
  Since each edge of $X(\mathcal{C})$ is contained in a positive facet,
  $\rho_{P}$ represents a translation in $W_{k}$. Hence $\rho_{P}$ also respects
  negative facets.
\end{proof}

\section[Colored $k$-configurations from planar difference sets]{Construction of colored $k$-configurations from planar difference sets}
\label{construction-planar}

In \cref{quotient}, we constructed simplicial complexes with specific properties
from colored $k$-configurations. In this section, we give a construction of
colored $k$-configurations from planar difference sets. As a special case, we
will obtain our generalization of the 7-vertex triangulation of $T^{2}$
(\cref{pg-complex}).

We begin by defining planar difference sets (see~\cite{beth1999design}):

\begin{definition}
\label{planar}
  Let $G$ be an abelian group. A \emph{planar difference set} in $G$ is a subset
  $A \subseteq G$, such that for each $g \in G$ other than the identity, there
  exist unique $a_{1}, a_{2} \in A$ with $g = a_{1} - a_{2}$. The \emph{order}
  of $A$ is $|A| - 1$.
\end{definition}
Note that $|G| = (k + 1) \cdot k + 1 = k^{2} + k + 1$. In fact, a planar
difference set forms a projective plane, in the following way:

\begin{lemma}[Singer~\cite{singer1938theorem}]
\label{planar-projective}
  Let $G$ be an abelian group, and let $A$ be a planar difference set in $G$
  with $|A| = k + 1$. Then we obtain a projective plane of order $k$, where:
  \begin{itemize}
    \item
    The point set $P$ is the set of elements of $G$.

    \item
    The line set $L$ is also the set of elements of $G$.

    \item
    A point $p \in P$ and line $l \in L$ are incident if $p - l \in A$.

  \end{itemize}
\end{lemma}

\begin{proof}
  Fix distinct points $p_{1}, p_{2} \in G$. Then a line $l \in G$ is incident
  with both of $p_{1}, p_{2}$ if and only if $p_{1} - l = a_{1}$ and $p_{2} - l
  = a_{2}$ for $a_{1}, a_{2} \in A$, implying $p_{1} - p_{2} = a_{1} - a_{2}$.
  Such $a_{1}, a_{2} \in A$ are uniquely determined by $p_{1}, p_{2}$, so any
  two points lie on exactly one common line. The dual statement holds similarly.
\end{proof}

For example, the projective planes $PG(2, \mathbb{F}_{q})$ all arise in this
way:
\begin{remark}[Singer~\cite{singer1938theorem}]
\label{pg-planar}
  Let $q$ be a prime power. Then the finite projective plane $PG(2,
  \mathbb{F}_{q})$ corresponds to a planar difference set in $\mathbb{Z}_{q^{2}
  + q + 1}$.
\end{remark}
However, it is not known whether all finite projective planes arising from
planar difference sets are isomorphic to $PG(2, \mathbb{F}_{q})$;
see~\cite{huang2009uniqueness} for a partial result in this direction. We now
give our construction.

\begin{theorem}
\label{planar-colored}
  Let $G$ be an abelian group, and let $A$ be a planar difference set in $G$
  with $|A| = k + 1$. The corresponding projective plane can be colored, giving
  a colored $(k + 1)$-configuration $\mathcal{C}$. Moreover, $G$ acts freely on
  $\mathcal{C}$.
\end{theorem}

\begin{proof}
  Let $\mathcal{C}$ be the projective plane described in
  \cref{planar-projective}. Then $\mathcal{C}$ is connected by \cref{connected},
  since any two points lie on a common line.

  We define a coloring $\chi \colon R(\mathcal{C}) \to A$ by $\chi(p, l) = p -
  l$; note that $p - l \in A$ since $p, l$ are incident in $\mathcal{C}$. To
  check the 6-cycle property, let $a, b, c\in A$, and consider the path $p_{0},
  l_{0}, p_{1}, l_{1}, p_{2}, l_{2}, p_{3}$ in $G(\mathcal{C})$ with
  \begin{align*}
    p_{0} - l_{0} &= a\\
    p_{1} - l_{0} &= b\\
    p_{1} - l_{1} &= c\\
    p_{2} - l_{1} &= a\\
    p_{2} - l_{2} &= b\\
    p_{3} - l_{2} &= c
  \end{align*}
  Taking an alternating sum gives $p_{0} - p_{3} = 0$, so the 6-cycle property
  holds.

  The free group action of $G$ on $\mathcal{C}$ is given by $g \cdot p = g + p$
  for points $p \in G$, and $g \cdot l = g + l$ for lines $l \in G$; this
  preserves incidences and colors.
\end{proof}

\begin{corollary}
\label{pg-colored}
  Let $q$ be a prime power. Then $PG(2, \mathbb{F}_{q})$ can be colored, giving
  a colored $(q + 1)$-configuration $\mathcal{C}$ with $q^{2} + q + 1$ points
  and $q^{2} + q + 1$ lines. Moreover, $\mathbb{Z}_{q^{2} + q + 1}$ acts freely
  on $\mathcal{C}$.
\end{corollary}

\begin{proof}
  By \cref{pg-planar} and \cref{planar-colored}.
\end{proof}

As a corollary, we get the following promised result:

\pgcomplex

\begin{proof}
  By \cref{pg-colored} and \cref{quotient}, since $PG(2, \mathbb{F}_{q})$ is
  self-dual.
\end{proof}

If we express $PG(2, \mathbb{F}_{q})$ as a planar difference set, the duality
between positive and negative facets becomes clear. For example, $PG(2,
\mathbb{F}_{2})$ corresponds to the planar difference set $\{0, 1, 3\}$ in
$\mathbb{Z}/7\mathbb{Z}$. The positive facets are translates of $\{0, 1, 3\}$,
and the negative facets are translates of $\{0, -1, -3\}$ (see \cref{fig:torus2}).

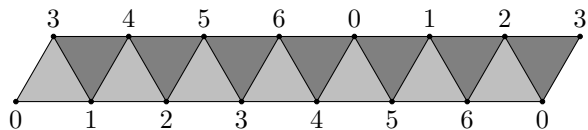
\begin{figure}
\begin{center}
\begin{tikzpicture}
  \path[fill = gray!50] (0, 0) -- (1, 0) -- (0.5, 0.866) -- (0, 0);
  \path[fill = gray!50] (1, 0) -- (2, 0) -- (1.5, 0.866) -- (1, 0);
  \path[fill = gray!50] (2, 0) -- (3, 0) -- (2.5, 0.866) -- (2, 0);
  \path[fill = gray!50] (3, 0) -- (4, 0) -- (3.5, 0.866) -- (3, 0);
  \path[fill = gray!50] (4, 0) -- (5, 0) -- (4.5, 0.866) -- (4, 0);
  \path[fill = gray!50] (5, 0) -- (6, 0) -- (5.5, 0.866) -- (5, 0);
  \path[fill = gray!50] (6, 0) -- (7, 0) -- (6.5, 0.866) -- (6, 0);

  \path[fill = gray] (0.5, 0.866) -- (1, 0) -- (1.5, 0.866) -- (0.5, 0.866);
  \path[fill = gray] (1.5, 0.866) -- (2, 0) -- (2.5, 0.866) -- (1.5, 0.866);
  \path[fill = gray] (2.5, 0.866) -- (3, 0) -- (3.5, 0.866) -- (2.5, 0.866);
  \path[fill = gray] (3.5, 0.866) -- (4, 0) -- (4.5, 0.866) -- (3.5, 0.866);
  \path[fill = gray] (4.5, 0.866) -- (5, 0) -- (5.5, 0.866) -- (4.5, 0.866);
  \path[fill = gray] (5.5, 0.866) -- (6, 0) -- (6.5, 0.866) -- (5.5, 0.866);
  \path[fill = gray] (6.5, 0.866) -- (7, 0) -- (7.5, 0.866) -- (6.5, 0.866);

  \draw[fill] (0, 0) circle [radius = 0.03];
  \draw[fill] (1, 0) circle [radius = 0.03];
  \draw[fill] (2, 0) circle [radius = 0.03];
  \draw[fill] (3, 0) circle [radius = 0.03];
  \draw[fill] (4, 0) circle [radius = 0.03];
  \draw[fill] (5, 0) circle [radius = 0.03];
  \draw[fill] (6, 0) circle [radius = 0.03];
  \draw[fill] (7, 0) circle [radius = 0.03];

  \draw[fill] (0.5, 0.866) circle [radius = 0.03];
  \draw[fill] (1.5, 0.866) circle [radius = 0.03];
  \draw[fill] (2.5, 0.866) circle [radius = 0.03];
  \draw[fill] (3.5, 0.866) circle [radius = 0.03];
  \draw[fill] (4.5, 0.866) circle [radius = 0.03];
  \draw[fill] (5.5, 0.866) circle [radius = 0.03];
  \draw[fill] (6.5, 0.866) circle [radius = 0.03];
  \draw[fill] (7.5, 0.866) circle [radius = 0.03];

  \draw (0, 0) -- (1, 0) -- (2, 0) -- (3, 0)
    -- (4, 0) -- (5, 0) -- (6, 0) -- (7, 0);
  \draw (0.5, 0.866) -- (1.5, 0.866) -- (2.5, 0.866) -- (3.5, 0.866)
    -- (4.5, 0.866) -- (5.5, 0.866) -- (6.5, 0.866) -- (7.5, 0.866);
  \draw (0, 0) -- (0.5, 0.866) -- (1, 0) -- (1.5, 0.866)
    -- (2, 0) -- (2.5, 0.866) -- (3, 0) -- (3.5, 0.866)
    -- (4, 0) -- (4.5, 0.866) -- (5, 0) -- (5.5, 0.866)
    -- (6, 0) -- (6.5, 0.866) -- (7, 0) -- (7.5, 0.866);

  \node [below] at (0, 0) {$0$};
  \node [below] at (1, 0) {$1$};
  \node [below] at (2, 0) {$2$};
  \node [below] at (3, 0) {$3$};
  \node [below] at (4, 0) {$4$};
  \node [below] at (5, 0) {$5$};
  \node [below] at (6, 0) {$6$};
  \node [below] at (7, 0) {$0$};

  \node [above] at (0.5, 0.866) {$3$};
  \node [above] at (1.5, 0.866) {$4$};
  \node [above] at (2.5, 0.866) {$5$};
  \node [above] at (3.5, 0.866) {$6$};
  \node [above] at (4.5, 0.866) {$0$};
  \node [above] at (5.5, 0.866) {$1$};
  \node [above] at (6.5, 0.866) {$2$};
  \node [above] at (7.5, 0.866) {$3$};
\end{tikzpicture}
\caption{The 7-vertex triangulation of $T^{2}$.}
\label{fig:torus2}
\end{center}
\end{figure}

\section[Colored $k$-configurations from commutative semifields]{Construction of colored $k$-configurations from commutative semifields}
\label{construction-semifields}

In this section, we give a construction of colored $k$-configurations from
commutative semifields, giving another possible input to our simplicial complex
construction (\cref{quotient}). A semifield has the properties of a field,
except that multiplication is not required to be associative or commutative:

\begin{definition}[\cite{knuth1963finite}; see~\cite{weibel2007survey}]
  A \emph{semifield} consists of a set $S$ and operations $(+), (\cdot)$, such
  that:
  \begin{itemize}
    \item
    $S$ forms a group under addition; we call the additive identity 0.

    \item
    If $a, b \in S$, $a \ne 0$, then there exists unique $x \in S$ with $a \cdot
    x = b$.

    \item
    If $a, b \in S$, $a \ne 0$, then there exists unique $y \in S$ with $y \cdot
    a = b$.

    \item
    If $a, b, c \in S$, then $a \cdot (b + c) = ab + ac$, and $(a + b) \cdot c =
    ac + bc$.

    \item
    There exists $1 \in S$ with $1 \ne 0$, and $1 \cdot a = a \cdot 1 = a$ for
    all $a \in S$.

  \end{itemize}
\end{definition}
This definition is important in the classification of projective planes. Any
projective plane admits a coordinatization with coordinates in a \emph{ternary
ring}; conversely, any ternary ring gives rise to a projective plane. If the
ternary ring is also a semifield, the corresponding projective plane has certain
symmetries. Moreover, if the ternary ring is also a field, the corresponding
projective plane is $PG(2, \mathbb{F}_{q})$. (See~\cite{weibel2007survey} for
definitions and details.)

Like fields, the number of elements in a semifield is always a prime power
(see~\cite{knuth1963finite}, Section~2.5). But there are commutative semifields
that are not fields; see~\cite{weibel2007survey} for examples. We outline the
construction of projective planes from semifields; this is a special case of the
general construction for ternary rings:

\begin{theorem}[Hall~\cite{hall1943projective}]
\label{semifield-projective}
  Let $F$ be a semifield with $|F| = q$. Then there exists a finite projective
  plane of order $q$.
\end{theorem}

\begin{proof}
  The point set $P$ consists of:
  \begin{itemize}
    \item
    Points $(x, y)$ for $x, y \in F$.

    \item
    Points $(x)$ for $x \in F \sqcup \{\infty\}$; this is a ``point at
    infinity'' for slope $x$.

  \end{itemize}
  The line set $L$ consists of:
  \begin{itemize}
    \item
    Lines $[a, b]$ for $a, b \in F$; this is the ``line'' $y = ax + b$.

    \item
    Lines $[a]$ for $a \in F \sqcup \{\infty\}$; this is the ``line'' $x = a$.

  \end{itemize}
  The set of incidence relations $R$ consists of:
  \begin{itemize}
    \item
    Incidences $((x, y), [a, b])$ for $x, y, a, b \in F$ with $y = ax + b$.

    \item
    Incidences $((x, y), [x])$ for $x, y \in F$.

    \item
    Incidences $((a), [a, b])$ for $a, b \in F$.

    \item
    Incidences $((x), [\infty])$ for $x \in F$.

    \item
    Incidences $((\infty), [a])$ for $a \in F$.

    \item
    The incidence $((\infty), [\infty])$.

  \end{itemize}
  We omit the proofs of the projective plane properties.
\end{proof}

We now give our construction.

\begin{theorem}
\label{semifield-colored}
  Let $F$ be a commutative semifield with $|F| = q$. Then there exists a colored
  $q$-configuration with $q^{2}$ points and $q^{2}$ lines.
\end{theorem}

\begin{proof}
  Starting with the finite projective plane from
  \cref{semifield-projective}, delete all points and lines other than
  those of the form $(x, y), [a, b]$. This deletes a line through each remaining
  point and a point on each remaining line; hence the result is a
  $q$-configuration $\mathcal{C}$ with $q^{2}$ points and $q^{2}$ lines.

  To show that $\mathcal{C}$ is connected, it suffices by \cref{connected} to
  show that there exists a path in $G(\mathcal{C})$ between any two points
  $(x_{1}, y_{1}), (x_{2}, y_{2})$. If $x_{1} \ne x_{2}$, then let $a \in F$
  satisfy $a \cdot (x_{2} - x_{1}) = y_{2} - y_{1}$, and let $b = y_{1} - a
  \cdot x_{1}$. Then $(x_{1}, y_{1})$ lies on $[a, b]$ by the definition of $b$,
  and $(x_{2}, y_{2})$ lies on $[a, b]$ by:
  \begin{align*}
    a \cdot x_{2} + b
      &= a \cdot x_{2} + (y_{1} - a \cdot x_{1})\\
      &= a \cdot (x_{2} - x_{1}) + y_{1}\\
      &= (y_{2} - y_{1}) + y_{1}\\
      &= y_{2}.
  \end{align*}
  Hence $(x_{1}, y_{1}), (x_{2}, y_{2})$ are adjacent if $x_{1} \ne x_{2}$.
  Since $F$ has at least two distinct elements $0, 1$, there exists a path in
  $G(\mathcal{C})$ between any two points of $\mathcal{C}$.

  Define a coloring $\chi \colon R(\mathcal{C}) \to F$ by $\chi((x, y), [a, b])
  = x + a$. In an incidence $((x, y), [a, b])$, each of $y, b$ are uniquely
  determined by the other three variables, which implies that $\chi$ is a
  $q$-edge coloring of $G(\mathcal{C})$. To check the 6-cycle property, let $c,
  d, e \in F$, and consider the path $(x_{0}, y_{0})$, $[a_{1}, b_{1}]$,
  $(x_{1}, y_{1})$, $[a_{2}, b_{2}]$, $(x_{2}, y_{2})$, $[a_{3}, b_{3}]$,
  $(x_{3}, y_{3})$ in $G(\mathcal{C})$ with
  \begin{align*}
    x_{0} + a_{1} &= c\\
    x_{1} + a_{1} &= d\\
    x_{1} + a_{2} &= e\\
    x_{2} + a_{2} &= c\\
    x_{2} + a_{3} &= d\\
    x_{3} + a_{3} &= e
  \end{align*}
  Taking an alternating sum gives $x_{0} - x_{3} = 0$; it remains to show $y_{0}
  = y_{3}$. Note that $a_{1} = c - x_{0}$, $a_{2} = c - d + e - x_{0}$, and
  $a_{3} = e - x_{0}$. We have
  \begin{align*}
    y_{3} - y_{0}
      &= (y_{3} - y_{2}) + (y_{2} - y_{1}) + (y_{1} - y_{0})\\
      &= a_{3} \cdot (x_{3} - x_{2}) + a_{2} \cdot (x_{2} - x_{1}) + a_{1} \cdot (x_{1} - x_{0})\\
      &= (e - x_{0}) \cdot (e - d) + (c - d + e - x_{0}) \cdot (c - e) + (c - x_{0}) \cdot (d - c)\\
      &= c \cdot d + d \cdot e + e \cdot c - d \cdot c - e \cdot d - c \cdot e\\
      &= 0.
  \end{align*}
  where we use commutativity in the last step. This completes the proof.
\end{proof}

In \cref{semifield-colored}, consider taking $F$ to be the field
$\mathbb{F}_{q}$. This gives a colored $q$-configuration $\mathcal{C}$; contrast
this with \cref{pg-colored}, which gives a colored $(q + 1)$-configuration
$\mathcal{C}'$. The underlying $q$-configuration of $\mathcal{C}$ can be
completed to that of $\mathcal{C}'$ by adding the points and lines at infinity,
but the coloring of $\mathcal{C}$ cannot be extended to $\mathcal{C}'$; in this
sense, the two constructions are distinct.

\section{Relationships with Sidon sets \& linear codes}
\label{relationship}

We now explain the connection between Sidon sets, linear codes, and colored
$k$-configurations. To begin, we define a Sidon set:

\begin{definition}[\cite{sidon1932satz}; see~\cite{kovacevic2019sidon}]
  Let $G$ be an abelian group, written additively. A set $B = \{b_{0}, b_{1},
  \ldots , b_{k}\} \subseteq G$ is a \emph{Sidon set of order $h$} if the sums
  $b_{i_{1}} + \cdots + b_{i_{h}}$, $0 \le i_{1} \le \cdots \le i_{h} \le k$,
  are all different.
\end{definition}

For $h = 2$, this condition is equivalent to the condition that the differences
$b_{i} - b_{j}$ for $i \ne j$ are all distinct, so $|G| \ge k^{2} + k + 1$.
Moreover:

\begin{remark}
\label{characterize-planar}
  Let $G$ be an abelian group with $|G| = k^{2} + k + 1$. Then a planar
  difference set in $G$ is exactly a Sidon set $B$ of order 2 with $|B| = k$.
\end{remark}

As with planar difference sets, Sidon sets may be translated; if $B$ is a Sidon
set of order $h$ in $G$, then so is $\{b + g : b \in B\}$ for any $g \in G$.

Next, we define a linear code:
\begin{definition}[see~\cite{kovacevic2019sidon}]
  A \emph{linear code} with radius $r$ in $A_{k}$ is a lattice $\mathcal{L}
  \subseteq A_{k}$, such that the balls $B_{r}(\vec{x}) = \{\vec{y} : d(\vec{x},
  \vec{y}) \le r\}$ for $\vec{x} \in \mathcal{L}$ are all disjoint. We say that
  $\mathcal{L}$ is \emph{perfect} if the balls $B_{r}(\vec{x})$ also cover
  $A_{k}$.
\end{definition}

There is a known correspondence between Sidon sets and linear codes; recall from
\cref{lattice} that $A_{n} \cong \mathbb{Z}^{n}$ on ignoring the last
coordinate:
\begin{theorem}[Kovačević~\cite{kovacevic2019sidon}]
\label{sidon-code}
  \quad
  \begin{enumerate}[(a)]
    \item
    Let $B = \{0, b_{1}, \ldots , b_{k}\}$ be a Sidon set of order 2 in an
    abelian group $G$, and suppose $B$ generates $G$. Then the lattice
    $$\mathcal{L} = \{\vec{x} \in A_{k} : \sum_{i = 1}^{k}x_{i}b_{i} = 0\}$$
    is a linear code with radius 1 in $A_{k}$, and $G \cong A_{k} /
    \mathcal{L}$. (Here $x_{i}b_{i}$ denotes the sum in $G$ of $|x_{i}|$ copies
    of $b_{i}$ if $x_{i} > 0$ and of $-b_{i}$ if $x_{i} < 0$.)

    \item
    Conversely, if $\mathcal{L}$ is a linear code with radius 1 in $A_{k}$, then
    the group $G = A_{k} / \mathcal{L}$ contains a Sidon set $B$ of order 2 with
    $|B| = k + 1$, such that $B$ generates $G$.
  \end{enumerate}
\end{theorem}

We also have a correspondence with colored $(k + 1)$-configurations:

\correspondence

\begin{proof}
  We first give a map $\eta$ from (2) to (3). Let $\mathcal{L}$ be a linear code
  with radius 1 in $A_{k}$, with $|A_{k} / \mathcal{L}| = n$. We define a
  colored $(k + 1)$-configuration $\mathcal{C}$:
  \begin{itemize}
    \item
    The point set $P$ is the set of elements of $A_{k} / \mathcal{L}$.

    \item
    The line set $L$ is also the set of elements of $A_{k} / \mathcal{L}$.

    \item
    A point $p \in P$ and line $l \in L$ are incident if $p - l = \vec{e}_{i} -
    \vec{e}_{k + 1}$ for some $i \in [k + 1]$; the color of the incidence is
    $i$.
  \end{itemize}
  Fix distinct points $p_{1}, p_{2} \in A_{k} / \mathcal{L}$. Then a line $l \in
  A_{k} / \mathcal{L}$ is incident with both of $p_{1}, p_{2}$ if and only if
  $p_{1} - l = \vec{e}_{i} - \vec{e}_{k + 1}$ and $p_{2} - l = \vec{e}_{j} -
  \vec{e}_{k + 1}$ for $i, j \in [k + 1]$, implying $p_{1} - p_{2} = \vec{e}_{i}
  - \vec{e}_{j}$, and $i \ne j$. Now the vectors $\vec{e}_{i} - \vec{e}_{j}$ for
  distinct $i, j \in [k + 1]$ have pairwise distance 1 or 2 in $A_{k}$, and
  hence are pairwise distinct in $A_{k} / \mathcal{L}$. Therefore, there is at
  most one choice of $i, j$ with $p_{1} - p_{2} = \vec{e}_{i} - \vec{e}_{j}$, so
  there is at most one $l$ incident with both of $p_{1}, p_{2}$. The dual
  statement holds similarly, so $\mathcal{C}$ is a $(k + 1)$-configuration.

  To show that $\mathcal{C}$ is connected, it suffices by \cref{connected} to
  show that there exists a path in $G(\mathcal{C})$ between any two points
  $p_{1}, p_{2} \in A_{k} / \mathcal{L}$. Suppose we have $p_{1} - p_{2} =
  \vec{e}_{i} - \vec{e}_{k + 1}$; then $p_{1}, p_{2}$ both lie on the line
  $p_{2}$. The general case follows by writing $p_{1} - p_{2}$ as an integer
  combination of vectors of the form $\vec{e}_{i} - \vec{e}_{k + 1}$.

  To show that $\mathcal{C}$ has the 6-cycle property, let $a, b, c \in [k +
  1]$, and consider the path $p_{0}, l_{0}, p_{1}, l_{1}, p_{2}, l_{2}, p_{3}$
  in $G(\mathcal{C})$ with
  \begin{align*}
    p_{0} - l_{0} &= \vec{e}_{a} - \vec{e}_{k + 1}\\
    p_{1} - l_{0} &= \vec{e}_{b} - \vec{e}_{k + 1}\\
    p_{1} - l_{1} &= \vec{e}_{c} - \vec{e}_{k + 1}\\
    p_{2} - l_{1} &= \vec{e}_{a} - \vec{e}_{k + 1}\\
    p_{2} - l_{2} &= \vec{e}_{b} - \vec{e}_{k + 1}\\
    p_{3} - l_{2} &= \vec{e}_{c} - \vec{e}_{k + 1}
  \end{align*}
  Taking an alternating sum gives $p_{0} - p_{3} = 0$, so the 6-cycle property
  holds.

  \bigskip

  \cref{sidon-code} gives a correspondence between (1) and (2), and
  \cref{construction} gives a map $\theta$ from (3) to (2), since if
  $\ell(\vec{0}) = p$, then by the properties of the lemma, $\ell^{-1}(p)$ is a
  linear code with radius 1. The construction in \cref{construction} involves a
  choice of $\vec{v}_{0} \in A_{k}$ and $p_{0} \in P(\mathcal{C})$, but distinct
  choices of $\vec{v}_{0}, p_{0}$ give the same linear code
  $\ell^{-1}(\ell(\vec{0}))$ by \cref{construction-choices}. Hence $\theta$ is
  well-defined, and we may choose $\vec{v}_{0}, p_{0}$ as we like. We have the
  following maps:

  \begin{center}
  \begin{tikzcd}[column sep=huge]
    \text{(1)} \arrow[r, shift left, "\text{(\cref{sidon-code})}"] &
    \text{(2)} \arrow[l, shift left, "\text{(\cref{sidon-code})}"]
               \arrow[r, shift left, "\text{$\eta$ (above)}"] &
    \text{(3)} \arrow[l, shift left, "\text{$\theta$ (\cref{construction})}"]
  \end{tikzcd}
  \end{center}

  It remains to show that $\eta$ and $\theta$ are inverses.

  \bigskip

  \noindent
  \emph{Claim.} Let $\mathcal{L}$ be a linear code with radius 1 in $A_{k}$,
  $\mathcal{C} = \eta(\mathcal{L})$, $\mathcal{L}' = \theta(\mathcal{C})$:
  \begin{center}
  \begin{tikzcd}
    \mathcal{L} \arrow[r, mapsto, "\eta"] &
    \mathcal{C} \arrow[r, mapsto, "\theta"] &
    \mathcal{L}'
  \end{tikzcd}
  \end{center}
  Then $\mathcal{L} = \mathcal{L}'$.

  \bigskip

  \noindent \emph{Proof of claim.} Recall from above that $P(\mathcal{C}) =
  A_{k} / \mathcal{L}$, so $\theta$ constructs a labeling $\ell \colon A_{k} \to
  A_{k} / \mathcal{L}$. Choose $\vec{v}_{0} = \vec{0}$, $p_{0} = \vec{0}$, so
  that $\ell(\vec{0}) = \vec{0}$. We claim $\ell(\vec{x}) = \vec{x}$ for all
  $\vec{x} \in A_{k}$, which we prove by induction on $d(\vec{0}, \vec{x})$.

  For points $p \in A_{k} / \mathcal{L}$, we have $\phi_{i}(p) = p - \vec{e}_{i}
  + \vec{e}_{k + 1}$; for lines $l \in A_{k} / \mathcal{L}$, we have
  $\phi_{i}(l) = p + \vec{e}_{i} - \vec{e}_{k + 1}$. Therefore, if
  $\ell(\vec{w}) = \vec{w}$ for $\vec{w} \in A_{k}$, then
  \begin{align*}
    \ell(\vec{w} - \vec{e}_{i} + \vec{e}_{j})
      &= (\phi_{j} \circ \phi_{i})(\ell(\vec{w}))\\
      &= (\phi_{j} \circ \phi_{i})(\vec{w})\\
      &= \vec{w} - \vec{e}_{i} + \vec{e}_{j}.
  \end{align*}
  By induction, $\ell(\vec{x}) = \vec{x}$ for all $\vec{x} \in A_{k}$. The
  linear code $\mathcal{L}'$ is $\ell^{-1}(\ell(\vec{0})) = \ell^{-1}(\vec{0})$,
  which is the kernel of $\ell \colon A_{k} \to A_{k} / \mathcal{L}$, which is
  $\mathcal{L}$. Hence $\mathcal{L} = \mathcal{L}'$.

  \bigskip

  \noindent
  \emph{Claim.} Let $\mathcal{C}$ be a colored $(k + 1)$-configuration,
  $\mathcal{L} = \theta(\mathcal{C})$, $\mathcal{C}' = \eta(\mathcal{L})$:
  \begin{center}
  \begin{tikzcd}
    \mathcal{C} \arrow[r, mapsto, "\theta"] &
    \mathcal{L} \arrow[r, mapsto, "\eta"] &
    \mathcal{C}'
  \end{tikzcd}
  \end{center}
  Then we have an isomorphism $\mathcal{C} \cong \mathcal{C}'$.

  \bigskip

  \noindent
  \emph{Proof of claim.} Let $\ell \colon A_{k} \to P(\mathcal{C})$ be the
  labeling constructed by $\theta$, and then let $\tilde{\ell} \colon A_{k} /
  \mathcal{L} \to P(\mathcal{C})$ be the isomorphism of sets guaranteed by the
  universal property of the quotient. Let $\tilde{\ell}^{-1} \colon
  P(\mathcal{C}) \to A_{k} / \mathcal{L}$ be the inverse of $\tilde{\ell}$.

  We claim $\tilde{\ell}^{-1}(\phi_{i}(l)) = \tilde{\ell}^{-1}(\phi_{j}(l)) +
  \vec{e}_{i} - \vec{e}_{j}$, for all $i, j \in [k + 1]$. To see this, take the
  identity $\ell(\vec{x} - \vec{e}_{i} + \vec{e}_{j}) = (\phi_{j} \circ
  \phi_{i})(\ell(\vec{x}))$, and let
  $$l
    = \phi_{j}(\ell(\vec{x} - \vec{e}_{i} + \vec{e}_{j}))
    = \phi_{i}(\ell(\vec{x})).$$
  Then we have $\vec{x} = \tilde{\ell}^{-1}(\phi_{j}(l)) + \vec{e}_{i} -
  \vec{e}_{j}$, and $\vec{x} = \tilde{\ell}^{-1}(\phi_{i}(l))$. Therefore, we
  have
  $$\tilde{\ell}^{-1}(\phi_{i}(l))
    = \tilde{\ell}^{-1}(\phi_{j}(l)) + \vec{e}_{i} - \vec{e}_{j}.$$

  Now we define a map $\mathcal{C} \to \mathcal{C}'$:
  \begin{itemize}
    \item
    Map $p \in P(\mathcal{C})$ to $\tilde{\ell}^{-1}(p)$ in $P(\mathcal{C}') =
    A_{k} / \mathcal{L}$.

    \item
    Map $l \in L(\mathcal{C})$ to $\tilde{\ell}^{-1}(\phi_{k + 1}(l))$ in
    $L(\mathcal{C}') = A_{k} / \mathcal{L}$.

  \end{itemize}
  Since $\tilde{\ell}^{-1}, \phi_{k + 1}$ are isomorphisms, these maps
  $P(\mathcal{C}) \to P(\mathcal{C}')$ and $L(\mathcal{C}) \to L(\mathcal{C}')$
  are isomorphisms. For incidences, we have the following equivalences:
  \begin{alignat*}{2}
    &\text{$p, l$ incident in $\mathcal{C}$}&&\\
    &\qquad\Leftrightarrow\quad p = \phi_{i}(l),\quad
      &&\text{some $i \in [k + 1]$}\\
    &\qquad\Leftrightarrow\quad \tilde{\ell}^{-1}(p) = \tilde{\ell}^{-1}(\phi_{i}(l)),\quad
      &&\text{some $i \in [k + 1]$}\\
    &\qquad\Leftrightarrow\quad \tilde{\ell}^{-1}(p) = \tilde{\ell}^{-1}(\phi_{k + 1}(l)) + \vec{e}_{i} - \vec{e}_{k + 1},\quad
      &&\text{some $i \in [k + 1]$}\\
    &\qquad\Leftrightarrow\quad\text{$\tilde{\ell}^{-1}(p), \tilde{\ell}^{-1}(\phi_{k + 1}(l))$ incident in $\mathcal{C'}$}&&
  \end{alignat*}
  Hence our map $\mathcal{C} \to \mathcal{C}'$ preserves both incidences and
  colors, which are given by the index $i$ above, so $\mathcal{C} \cong
  \mathcal{C}'$ as desired.

  \bigskip

  By the claims above, $\eta, \theta$ are inverses up to isomorphism.
\end{proof}

Now let $\mathcal{C}$ be the colored $(q + 1)$-configuration obtained from a
planar difference set via \cref{planar-colored}. Then in the theorem above, we
have $G \cong \mathbb{Z}_{q^{2} + q + 1}$, $|B| = q + 1$, and $\mathcal{L}$
perfect. This case of the correspondence between (1) and (2) is discussed in
Section 3 of~\cite{kovacevic2019sidon}; in particular, we recover a result of
Singer on Sidon sets (\cite{singer1938theorem}; see~\cite{bose1962theorems}). We
also obtain the following corollary:

\begin{corollary}
\label{correspondence-projective}
  We have a one-to-one correspondence (up to isomorphism) between the following
  two structures:
  \begin{enumerate}[(1)]
    \item
    Pairs $(G, A)$, where $G$ is an abelian group, and $A$ is a planar
    difference set in $G$ with $|A| = k$.

    \item
    Colored $k$-configurations $\mathcal{C}$ which are also projective
    planes.

  \end{enumerate}
\end{corollary}

\begin{proof}
  Replace $k$ with $k + 1$, and restrict the correspondence between (1) and (3)
  in \cref{correspondence} to the case $n = k^{2} + k + 1$. Then the result
  follows from the characterizations given by \cref{characterize-projective} and
  \cref{characterize-planar}.
\end{proof}

\section{Directions for further research}

By \cref{correspondence-projective}, we can restate open problems on planar
difference sets. For example, the conjecture that all planar difference sets are
cyclic becomes:

\cyclic

The conjecture that all planar difference sets are Desarguesian becomes:

\begin{conjecture}
  Let $\mathcal{C}$ be a colored $k$-configuration which is also a projective
  plane. Then $\mathcal{C}$ is Desarguesian (that is, isomorphic to $PG(2,
  \mathbb{F}_{q})$).
\end{conjecture}

More generally, we can ask the following existence question:

\begin{question}
  Which connected $k$-configurations $\mathcal{C}$ can be colored to form a
  colored $k$-configuration?
\end{question}

In a different direction, our simplicial complex construction $X(\mathcal{C})$
(\cref{quotient}) suggests possible topological obstructions to the existence of
not only Sidon sets and linear codes, but also planar difference sets and finite
projective planes. One possible obstruction is via the following conjecture:

\minimumcyclic

Along with \cref{planar-colored}, this conjecture implies that every cyclic
planar difference set has prime power order, an open problem in design theory
(see~\cite{beth1999design}, Chapter~VII). The free $\mathbb{Z}_{n}$-action
condition is necessary, since there exist simplicial complexes $X$ with
$\pi_{1}(X) \cong \mathbb{Z}^{k}$ on only $O(k)$
vertices~\cite{frick2021vertex}.

\section{Acknowledgements}

We thank Florian Frick for showing the author what became the motivating
observation of this work, that the 7-vertex triangulation of $T^{2}$ consists of
two copies of $PG(2, \mathbb{F}_{q})$, and for many helpful conversations
throughout.


\printbibliography

\end{document}